\documentclass[11pt]{article}

\usepackage{graphicx}
\usepackage{a4wide}
\usepackage[T1]{fontenc}
\usepackage[utf8]{inputenc}
\usepackage{amsmath, amssymb, amsthm}
\usepackage[english]{babel}
\usepackage{eurosym}
\usepackage{tikz}
\usepackage{verbatim,refcount}
\usepackage{fancybox,soul,color}
\usepackage{url}
\usepackage{hyperref}

\usepackage{subfig}
\usepackage{enumerate}
\usepackage{epsfig}
\usepackage[active]{srcltx}
\usepackage{amsfonts}
\usepackage{amsmath}
\usepackage[mathlines]{lineno}
\usepackage{xspace}
\usepackage{tikz}
\usepackage{multirow,tabu}
\usepackage{slashed}
\usepackage{todonotes}
\usepackage{eucal}

\usetikzlibrary{arrows,shapes,automata,backgrounds,petri}

\newtheorem{lemma}{Lemma}

\newtheorem{theorem}{Theorem}
\newtheorem{claim}{Claim}
\newtheorem{corollary}{Corollary}
\newtheorem{conjecture}{Conjecture}

\newtheorem{obs}{Observation}

\def\claimb{\vcenter\bgroup\advance\hsize by -8em\noindent
\refstepcounter{claimb}\ignorespaces\it}
\makeatletter
\def\endclaimb{\rm\egroup\leqno(\theclaim)\global\@ignoretrue}
\makeatother

    {\noindent \emph{Proof.} {}{#1}{}}{\hfill
    $\Diamond$\vspace{1em}}

\newcounter{rulecnt}
\newcommand{\rrule}[2]{\medskip\noindent{\bf \refstepcounter{rulecnt}\label{#1}Rule~\ref{#1}}  #2\ }%
\newcommand{\rxrule}[1]{\noindent{Rule~\ref{#1}}}%
\newcounter{confcnt}
\newcommand{\rconf}[2]{\noindent{\bf \refstepcounter{confcnt}\label{#1}($C_{\ref{#1}}$)}  #2\ }%
\newcommand{\rxconf}[1]{\noindent{\textup{(}\textbf{$C_{\ref{#1}}$}\textup{)}}}%

\newcommand{\ceil}[1]{\left\lceil #1 \right\rceil}

\def\S{\mathcal{S}}
\def\Pp{\mathcal{P}}
\def\Cc{\mathcal{C}}
\def\ch{\text{ch}}
\def\G{\mathcal{G}}

\begin{document}

\title{Partitioning edges of a planar graph into linear forests\\ and a matching}

\date{}

\author{Marthe Bonamy\thanks{LaBRI, University of Bordeaux, France} \and Jadwiga Czy\.zewska\thanks{University of Warsaw, Poland} \and \L ukasz Kowalik\footnotemark[2] \and Micha\l \ Pilipczuk\footnotemark[2]}

\maketitle

\begin{abstract}
We show that the edges of any planar graph of maximum degree at most $9$ can be partitioned into $4$ linear forests and a matching.
Combined with known results, this implies that the edges of any planar graph $G$ of odd maximum degree $\Delta\ge 9$ can be partitioned into $\tfrac{\Delta-1}{2}$ linear forests and one matching.
This strengthens well-known results stating that graphs in this class have chromatic index $\Delta$ [Vizing, 1965] and linear arboricity at most $\ceil{(\Delta+1)/2}$ [Wu, 1999].
\end{abstract}

\section{Introduction}

A {\em linear forest} is a
forest in which every connected component is a path. The linear arboricity
${\rm la}(G)$ of a graph $G$ is the minimum number of linear forests in $G$, whose union is the whole $G$.
The parameter was introduced by Harary~\cite{harary} in 1970 and determining its value for various graph classes is a vital area of research by today.

Note that since matchings are linear forests, linear arboricity lies between the chromatic index (partitioning into matchings) and the arboricity (partitioning into forests).
The connection with chromatic index can be deeper than it seems, namely it is conjectured that linear arboricity enjoys an analogue of Vizing's theorem.
This analogue is called Linear Arboricity Conjecture~\cite{AEH, alon_ijm} and states that for every graph $G$, we have $\ceil{\tfrac{\Delta}{2}}\le{\rm la}(G)\le\ceil{\tfrac{\Delta+1}{2}}$ (the lower bound being trivial).
The conjecture has been proved for  $\Delta\in\{3,4,5,6,8,10\}$~\cite{AEH,akiyama2,EP,Guldan}, for complete bipartite graphs~\cite{AEH}, for planar graphs~\cite{Wu,WW} and recently for 3-degenerate graphs~\cite{BasavarajuBFP20}.

In this paper we focus on planar graphs.
Cygan et al.~\cite{CyganHKLW12} proved the following theorem.

\begin{theorem}[\cite{CyganHKLW12}]\label{th:cygan}
	For any planar graph $G$ with maximum degree $\Delta \geq 9$, we have ${\rm la}(G)=\lceil\tfrac{\Delta}{2}\rceil$.
\end{theorem}

They have posed the following conjecture.

\begin{conjecture}[Planar Linear Arboricity Conjecture~\cite{CyganHKLW12}]
	\label{plac}
	For any planar graph $G$ of maximum degree $\Delta\ge 5$, we have ${\rm la}(G)=\ceil{\tfrac{\Delta}{2}}$.
\end{conjecture}

Note that Conjecture~\ref{plac} implies the Vizing Planar Graph Conjecture~\cite{Vizing65}, stating that any planar graph $G$ of maximum degree $\Delta\ge 6$ has chromatic index $\Delta$ (currently open only for $\Delta=6$).

The equality ${\rm la}(G)=\ceil{\tfrac{\Delta}{2}}$ holds for all odd values of $\Delta$, which follows from the weaker upper bound ${\rm la}(G)\le\ceil{\tfrac{\Delta+1}{2}}$~\cite{Wu,WW}.
Clearly, the odd case gives much more freedom, since then for each vertex $v$ there is at least one linear forest which has at most one edge incident with $v$.
The motivating question of this work is {\em Can we use this freedom to get an even tighter result?}
More specifically, we conjecture the following.

\begin{conjecture}
	\label{con:new}
	For every planar graph $G$ of odd maximum degree $\Delta\ge 7$ the edges of $G$ can be partitioned into $\tfrac{\Delta-1}{2}$ linear forests and one matching.
\end{conjecture}

Note that Conjecture~\ref{con:new} cannot be extended to $\Delta\in\{3,5\}$.
This is because Conjecture~\ref{con:new} implies that the graph under consideration has chromatic index $\Delta$ (color each linear forests with two colours and the matching forms the last colour).
However, the 4-clique and the icosahedron, each with one edge subdivided, are well-known to have chromatic index $\Delta+1$ ~\cite{Vizing65}.

On the other hand, known results imply Conjecture~\ref{con:new} for $\Delta\ge 11$.
Indeed, pick a $\Delta$-edge colouring of $G$ which exists by a theorem of Vizing~\cite{Vizing65}.
Let $M$ be the matching formed by an arbitrary colour in the colouring.
Then $G-M$ has maximum degree $\Delta-1$, which is even.
The claim follows by combining the matching $M$ with the partition of $G-M$ into $(\Delta(G)-1)/2$ linear forests, obtained using Theorem~\ref{th:cygan}.

Thus, the only two missing cases of Conjecture~\ref{con:new} are $\Delta=7$ and $\Delta=9$. In this work, we resolve the latter, as follows.

\begin{theorem}\label{th:main}
For any planar graph $G$ with maximum degree $\Delta(G) \leq 9$, we can partition the edges of $G$ into four linear forests and a matching.
\end{theorem}

Note that Theorem~\ref{th:main} strengthens well-known results stating that graphs in this class have chromatic index $\Delta$ [Vizing, 1965] and linear arboricity at most $\ceil{(\Delta+1)/2}$ [Wu 1999]. By combining Theorem~\ref{th:main} with Theorem~\ref{th:cygan} and the argument for the case $\Delta\ge 11$ of Conjecture~\ref{con:new} described above, we get the following corollary.

\begin{corollary}
Let $G$ be a planar graph with maximum degree $\Delta(G) \geq 9$.
If $\Delta(G)$ is even, the edges of $G$ can be partitioned into $\Delta(G)/2$ linear forests.
If $\Delta(G)$ is odd, the edges of $G$ can be partitioned into $(\Delta(G)-1)/2$ linear forests and a matching.
\end{corollary}

Theorem~\ref{th:main} is proved with the discharging method, by now a standard tool in colouring of planar graphs. It can be sketched as follows. We find a set of 10 configurations which cannot appear in a minimal counterexample, i.e., are {\em reducible}. Then, we show that Euler's formula implies that a minimal counterexample has to contain one of these configurations.
Unfortunately, the nature of our kind of colouring reveals much less symmetry than, for example, edge colouring or `pure' linear arboricity. This is because in our case a colour forms either a linear forest or a matching. As a result, the reducibility proofs for configurations become even more lengthy and tedious than usual. This is why here we include traditional hand-made proofs for only two configurations, while the reducibility of all the remaining ones is verified by a computer program.

The program has been extensively tested for a number of configurations. We compared results of two different implementations, one in C++, and another in Python.
The Python implementation (150 lines of code, excluding comments) was optimized for readability and is publicly available in a Github repository \url{https://github.com/lkowalik/linear-arboricity}. We give a detailed description of the program in Section~\ref{sec:computer} (and in the comments included in the code).

\section{Preliminaries}
All graphs in this paper are simple, i.e., do not contain multiple edges or loops. By a triangle we mean a cycle of length $3$.

For a graph $G$ and a vertex $v$ of $G$, by $d_G(v)$ we denote the degree of $v$ in $G$, and we omit the subscript when it is clear from the context.
A vertex of degree $d$ is called a {\em $d$-vertex}. Notation {\em $d^+$-vertex} (resp. {\em $d^-$-vertex}) means that this vertex is of degree at least (resp. at most) $d$.

A neighbour of degree $d$ adjacent to a vertex $v$ is called a {\em $d$-neighbour} of $v$. A neighbour of degree at least (resp. at most) $d$ to a vertex $v$ is called a {\em $d^+$-neighbour} (resp. {\em $d^-$-neighbour}) of $v$. Given two vertices $u$ and $v$, the vertex $u$ is a \emph{weak neighbour} (resp. \emph{semi-weak}) of $v$ if the edge $(u,v)$ belongs to exactly $2$ (resp. exactly $1$) triangles.

As $\ell(f)$ we denote the number of edges incident to a face $f$. A face is $big$ if $\ell(f)\geq 4$.


 A {\em facial walk} $w$ corresponding
to a face $f$ is the shortest closed walk induced by all edges incident with $f$.

Let $f$ be a face and let $v_0, v_1, \ldots, v_{\ell}$ be the facial walk of $f$, $v_{\ell}=v_0$. Then, for every $i=0,\ldots,\ell-1$, the triple $s=(v_{i-1}, v_i, v_{i+1})$ will be called a {\em $v_i$-segment} or just {\em segment} (indices considered modulo $\ell$).
The {\em length} of $s$ is defined as the length of $f$.
We say that $v_i$ is {\em incident} to $s$. We stress that $v_{i-1}$, $v_{i+1}$ are not incident to $s$.
A \emph{triangular segment} is a segment $s = (x, y, z)$ where $x$ and $z$ are connected with an edge.

Note that if $v$ is not a cutvertex, there is a one-to-one correspondence between $v$-segments and faces incident to $v$. When $v$ is a cutvertex, a $v$-segment $(x,v,y)$ can be thought of as a region of the face incident with edges $vx$ and $vy$.

\section{Proof of Theorem~\ref{th:main}}\label{sect:thm}
Let $\G$ be a family of minimal counterexamples, i.e., $G \in \G$ if $G$ is a simple planar graph with $\Delta(G) \leq 9$ whose edges cannot be partitioned into four linear forests and a matching and, among such graphs, $G$ has the minimum possible number of edges.

\subsection{Structure of a minimal counterexample}. We define configurations \rxconf{c:edge} to \rxconf{c:2neighbour33bigface} (see Figure~\ref{fig:config}).
\begin{itemize}
\item \rconf{c:edge}{is an edge $uv$ with $d(u)+d(v) \leq 10$.}
\item \rconf{c:kite}{is a triangular segment $(u,v,w)$ where $u$ has another neighbour $x$ such that $d(v)=d(x)=11-d(u)$.}
\item \rconf{c:2degree2neighbours}{is a vertex with two $2$-neighbours.}
\item \rconf{c:2consweak3neighbours}{is a vertex $u$ with neighbors $x$, $y$, $w$, $s$, $t$, where $d(y)=d(s)=3$, and $xywst$ is a path.}
\item \rconf{c:2weak3neighbours3neighbour}{is a vertex $u$ with three $3$-neighbours, from which at least two are weak and $u$ is the only common neighbour of any pair of weak $3$-neighbours of $u$.}
\item \rconf{c:weak3degree2}{is a vertex $u$ with both a weak $3$-neighbour and a $2$-neighbour.}
\item \rconf{c:8weak3degree4}{is a vertex $u$ of degree $8$ with both a weak $3$-neighbour and a $4$-neighbour.}
\item \rconf{c:smalltriangles}{is a triangle $(u, v, w)$ where $d(u)\le 5$, $d(v)\le 6$, $d(w)\leq 8$.}
\item \rconf{c:23triangles-3}{is a vertex $u$ with a $3$-neighbours $w$ and $z$, a $2$-neighbour $x$, and a neighbor $y$ which is adjacent to both $z$ and $x$.}
\item \rconf{c:2neighbour33bigface}{is a vertex $u$ with neighbors $v$, $x$, $y$, $z$, $s$, $t$, where $d(v)=2$, $d(y)=d(s)=3$, $y$ and $s$ share a common neigbour $z$ other than $u$, $x$ is adjacent to $y$ and $t$ is adjacent to $s$.}
\end{itemize}

\begin{figure}[!ht]
\centering
\subfloat[][\textbf{$\rxconf{c:edge}$}]{
\centering
\begin{tikzpicture}[scale=0.95]
\tikzstyle{whitenode}=[draw,circle,fill=white,minimum size=8pt,inner sep=0pt]
\tikzstyle{blacknode}=[draw,circle,fill=black,minimum size=6pt,inner sep=0pt]
\tikzstyle{tnode}=[draw,ellipse,fill=white,minimum size=8pt,inner sep=0pt]
     \draw (2,-3) node[whitenode] (u) [label=left:$u$] {}
        -- ++(360:1cm) node[whitenode] (v) [label=right:$v$] {};
    \node[right=0pt] at (1.4,-3.5) {\tiny{$d(u)+d(v)\leq 10$}};
\end{tikzpicture}
\label{fig:cc1}
}
\qquad
\subfloat[][\textbf{$\rxconf{c:kite}$}]{
\centering
\begin{tikzpicture}[scale=0.95]
\tikzstyle{whitenode}=[draw,circle,fill=white,minimum size=8pt,inner sep=0pt]
\tikzstyle{blacknode}=[draw,circle,fill=black,minimum size=6pt,inner sep=0pt]
\tikzstyle{tnode}=[draw,ellipse,fill=white,minimum size=8pt,inner sep=0pt]
\tikzstyle{texte} =[fill=white, text=black]
     \draw (2,-3) node[whitenode] (u) [label=right:$u$] {}
        -- ++(60:1cm) node[whitenode] (v) [label=right:$v$] {}
        -- ++(180:1cm) node[whitenode] (w) [label=left:$w$] {};
    \draw (1,-3) node[whitenode] (x)[label=left:$x$] {};
    \draw (x) edge []  node [label=left:] {} (u);
    \draw (w) edge []  node [label=left:] {} (u);
    \node[right=0pt] at (0,-3.5) {\tiny{$d(v)=d(x)=11-d(u)$}};
\end{tikzpicture}
\label{fig:cc2}
}
\qquad
\subfloat[][\textbf{$\rxconf{c:2degree2neighbours}$}]{
\centering
\begin{tikzpicture}[scale=0.95]
\tikzstyle{whitenode}=[draw,circle,fill=white,minimum size=8pt,inner sep=0pt]
\tikzstyle{blacknode}=[draw,circle,fill=black,minimum size=6pt,inner sep=0pt]
\tikzstyle{tnode}=[draw,ellipse,fill=white,minimum size=8pt,inner sep=0pt]
\tikzstyle{texte} =[fill=white, text=black]
     \draw (2,-3) node[whitenode] (u1) [label=above:$u$] {\tiny{2}}
        -- ++(180:1cm) node[whitenode] (v) [label=above:$v$] {}
        -- ++(180:1cm) node[whitenode] (u2) [label=above:$w$] {\tiny{2}};
\end{tikzpicture}
\label{fig:cc3}
}
\qquad
\subfloat[][\textbf{$\rxconf{c:2consweak3neighbours}$}]{
\centering
\begin{tikzpicture}[scale=0.95]
\tikzstyle{whitenode}=[draw,circle,fill=white,minimum size=8pt,inner sep=0pt]
\tikzstyle{blacknode}=[draw,circle,fill=black,minimum size=6pt,inner sep=0pt]
\tikzstyle{tnode}=[draw,ellipse,fill=white,minimum size=8pt,inner sep=0pt]
\tikzstyle{texte} =[fill=white, text=black]
     \draw (2,-3) node[whitenode] (u) [label=above:$u$] {}
        -- ++(60:1cm) node[whitenode] (x1) [label=above:$x$] {}
        -- ++(180:1cm) node[whitenode] (v1) [label=above:$y$] {\tiny{$3$}}
        -- ++(240:1cm) node[whitenode] (x2) [label=left:$w$] {}
        -- ++(300:1cm) node[whitenode] (v2) [label=left:$s$] {\tiny{$3$}}
        -- ++(0:1cm) node[whitenode] (x3) [label=right:$t$]{};
    \draw (v1) edge []  node [label=left:] {} (u);
    \draw (v2) edge []  node [label=left:] {} (u);
    \draw (x2) edge []  node [label=left:] {} (u);
    \draw (x3) edge []  node [label=left:] {} (u);
\end{tikzpicture}
\label{fig:cc4}
}

\qquad
\subfloat[][\textbf{$\rxconf{c:2weak3neighbours3neighbour}$}]{
\centering
\begin{tikzpicture}[scale=0.95]
\tikzstyle{whitenode}=[draw,circle,fill=white,minimum size=8pt,inner sep=0pt]
\tikzstyle{blacknode}=[draw,circle,fill=black,minimum size=6pt,inner sep=0pt]
\tikzstyle{tnode}=[draw,ellipse,fill=white,minimum size=8pt,inner sep=0pt]
\tikzstyle{texte} =[fill=white, text=black]
     \draw (2,-3) node[whitenode] (u) [label=left:$u$] {}
        -- ++(30:1cm) node[whitenode] (x1) [label=right:$x$] {}
        -- ++(150:1cm) node[whitenode] (v1) [label=above:$y$] {\tiny{$3$}}
        -- ++(210:1cm) node[whitenode] (x2) [label=left:$w$] {}
        -- ++(330:1cm) node[whitenode] (u2) [] {}
        -- ++(210:1cm) node[whitenode] (x3) [label=left:$z$] {}
        -- ++(330:1cm) node[whitenode] (v2) [label=right:$s$] {\tiny{$3$}}
        -- ++(30:1cm) node[whitenode] (x4) [label=right:$t$] {};

    \draw (3,-3) node[whitenode] (v3) [] {\tiny{3}};
    \draw (v1) edge []  node [label=left:] {} (u);
    \draw (v2) edge []  node [label=left:] {} (u);
    \draw (v3) edge []  node [label=left:] {} (u);
    \draw (x4) edge []  node [label=left:] {} (u);
    \draw (x4) edge []  node [label=left:] {} (v2);
    \draw (x3) edge []  node [label=left:] {} (v2);
    \draw (x1) edge []  node [label=left:] {} (v1);
    \draw (x2) edge []  node [label=left:] {} (v1);
\end{tikzpicture}
\label{fig:cc5}
}
\qquad
\subfloat[][\textbf{$\rxconf{c:weak3degree2}$}]{
\centering
\begin{tikzpicture}[scale=0.95]
\tikzstyle{whitenode}=[draw,circle,fill=white,minimum size=8pt,inner sep=0pt]
\tikzstyle{blacknode}=[draw,circle,fill=black,minimum size=6pt,inner sep=0pt]
\tikzstyle{tnode}=[draw,ellipse,fill=white,minimum size=8pt,inner sep=0pt]
\tikzstyle{texte} =[fill=white, text=black]
     \draw (2,-3) node[whitenode] (u) [label=above:$u$] {}
        -- ++(60:1cm) node[whitenode] (x1) [label=right:$x$] {}
        -- ++(180:1cm) node[whitenode] (v1) [label=above:$y$] {\tiny{$3$}}
        -- ++(240:1cm) node[whitenode] (x2) [label=left:$z$] {};
    \draw (3,-3) node[whitenode] (v2) [label=right:$w$] {\tiny{2}};
    \draw (v1) edge []  node [label=left:] {} (u);
    \draw (x2) edge []  node [label=left:] {} (u);
    \draw (v2) edge []  node [label=left:] {} (u);
\end{tikzpicture}
\label{fig:cc6}
}
\qquad
\subfloat[][\textbf{$\rxconf{c:8weak3degree4}$}]{
\centering
\begin{tikzpicture}[scale=0.95]
\tikzstyle{whitenode}=[draw,circle,fill=white,minimum size=8pt,inner sep=0pt]
\tikzstyle{blacknode}=[draw,circle,fill=black,minimum size=6pt,inner sep=0pt]
\tikzstyle{tnode}=[draw,ellipse,fill=white,minimum size=8pt,inner sep=0pt]
\tikzstyle{texte} =[fill=white, text=black]
     \draw (2,-3) node[whitenode] (u) [label=above:$u$] {\tiny{8}}
        -- ++(60:1cm) node[whitenode] (x1) [label=right:$x$] {}
        -- ++(180:1cm) node[whitenode] (v1) [label=above:$y$] {\tiny{$3$}}
        -- ++(240:1cm) node[whitenode] (x2) [label=left:$z$] {};
    \draw (3,-3) node[whitenode] (v2) [label=right:$w$] {\tiny{4}};
    \draw (v1) edge []  node [label=left:] {} (u);
    \draw (x2) edge []  node [label=left:] {} (u);
    \draw (v2) edge []  node [label=left:] {} (u);
\end{tikzpicture}
\label{fig:cc7}
}
\qquad
\subfloat[][\textbf{$\rxconf{c:smalltriangles}$}]{
\centering
\begin{tikzpicture}[scale=0.95]
\tikzstyle{whitenode}=[draw,circle,fill=white,minimum size=8pt,inner sep=0pt]
\tikzstyle{blacknode}=[draw,circle,fill=black,minimum size=6pt,inner sep=0pt]
\tikzstyle{tnode}=[draw,ellipse,fill=white,minimum size=8pt,inner sep=0pt]
\tikzstyle{texte} =[fill=white, text=black]
     \draw (2,-3) node[whitenode] (u) [label=right:$u$] {\tiny{$5$}}
        -- ++(60:1cm) node[whitenode] (v) [label=right:$v$] {\tiny{$6$}}
        -- ++(180:1cm) node[whitenode] (w) [label=left:$w$] {\tiny{$8$}};
    \draw (w) edge [] node [label=left:] {} (u);
\end{tikzpicture}
\label{fig:cc8}
}
\qquad
\subfloat[][\textbf{$\rxconf{c:23triangles-3}$}]{
\centering
\begin{tikzpicture}[scale=0.95]
\tikzstyle{whitenode}=[draw,circle,fill=white,minimum size=8pt,inner sep=0pt]
\tikzstyle{blacknode}=[draw,circle,fill=black,minimum size=6pt,inner sep=0pt]
\tikzstyle{tnode}=[draw,ellipse,fill=white,minimum size=8pt,inner sep=0pt]
\tikzstyle{texte} =[fill=white, text=black]
    \draw (2,-3) node[whitenode] (z) [label=left:$w$] {\tiny{3}}
        -- ++(0:1cm) node[whitenode] (u) [label=above:$u$] {}
        -- ++(320:1cm) node[whitenode] (v) [label=right:$x$] {\tiny{$2$}}
        -- ++(40:1cm) node[whitenode] (x1) [label=right:$y$] {}
        -- ++(140:1cm) node[whitenode] (w) [label=right:$z$] {\tiny{3}};
    \draw (x1) edge [] node [label=left:] {} (u);
    \draw (w) edge [] node [label=left:] {} (u);
\end{tikzpicture}
\label{fig:cc9}
}
\qquad
\subfloat[][\textbf{\rxconf{c:2neighbour33bigface}}]{
\centering
\begin{tikzpicture}[scale=0.95]
\tikzstyle{whitenode}=[draw,circle,fill=white,minimum size=8pt,inner sep=0pt]
\tikzstyle{blacknode}=[draw,circle,fill=black,minimum size=6pt,inner sep=0pt]
\tikzstyle{tnode}=[draw,ellipse,fill=white,minimum size=8pt,inner sep=0pt]
\tikzstyle{texte} =[fill=white, text=black]
    \draw (1,-3) node[whitenode] (v) [label=left:$v$] {\tiny{2}}
        -- ++(0:1cm) node[whitenode] (u) [label=above:$u$] {}
        -- ++(60:0.87cm) node[whitenode] (x1) [label=above:$x$] {}
        -- ++(330:0.5cm) node[whitenode] (x2) [label=right:$y$] {\tiny{$3$}}
        -- ++(330:1cm) node[whitenode] (x3) [label=right:$z$] {}
        -- ++(210:1cm) node[whitenode] (x4) [label=right:$s$] {\tiny{$3$}}
        -- ++(210:0.5cm) node[whitenode] (x5) [label=below:$t$] {};
    \draw (u) edge [] node [label=left:] {} (x2);
    \draw (u) edge [] node [label=left:] {} (x5);
    \draw (u) edge [] node [label=left:] {} (x4);
\end{tikzpicture}
\label{fig:cc10}
}
\caption{\label{fig:config} Reducible configurations. The vertices drawn above are pairwise different.}
\end{figure}
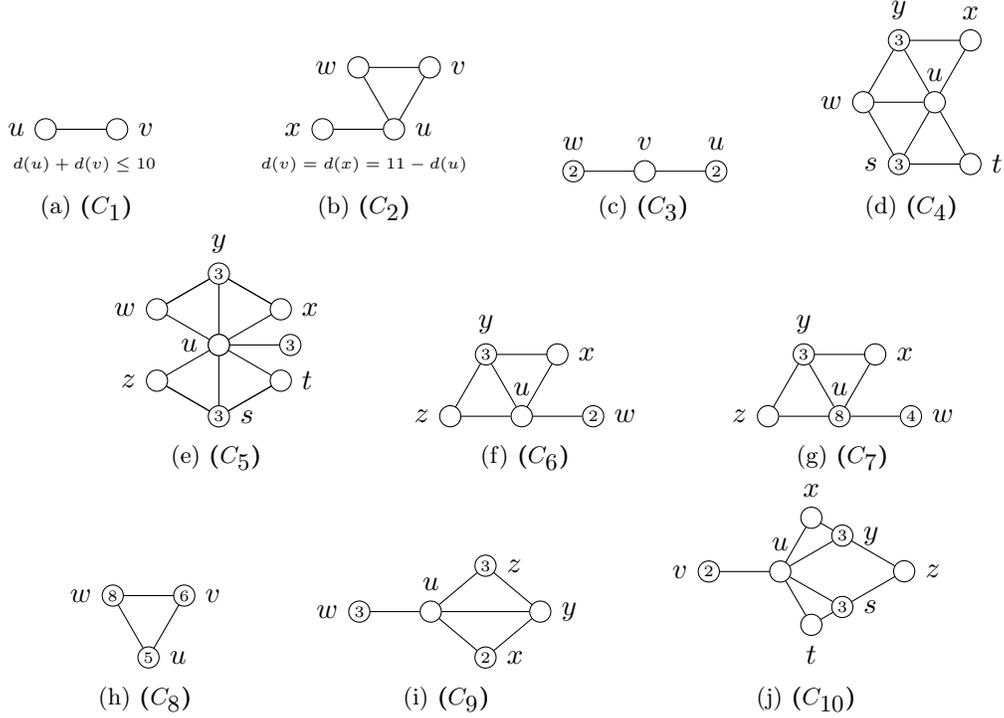

We stress here that configurations are defined as graphs (think: subgraphs of $G \in \G$) and not as plane embeddings of graphs. Hence, for example when we write that $G$ contains \rxconf{c:kite} it does not imply that the triangle $(u,v,w)$ bounds a triangular face.


We now present human-made proofs that configurations \rxconf{c:edge} and  \rxconf{c:2degree2neighbours} are reducible. In these proofs we use some common notation as follows.
We treat partitions to linear forests and a matching as colourings, i.e., functions $c: E(G)\rightarrow \{0,\ldots , 4\}$ such that for $i=1, \ldots , 4$ we have that $c^{-1}(i)$ is a linear forest and $c^{-1}(0)$ is a matching. By $d_i(c,v)$ we denote the number of edges which are incident with $v$ and coloured with $i$ in the colouring $c$.

\begin{lemma}\label{lem:c1}
	For every $G \in \G$, graph $G$ does not contain configuration \rxconf{c:edge}.
\end{lemma}
\begin{proof}
	Assume for a contradiction that there is an edge $uv$ such that $d(u)+d(v)\le 10$.
	Let $G'$ be the planar graph obtained from $G$ by removing edge $uv$.
	Since $G$ is a minimal counterexample, there is a colouring $c'$ of $G'$.

	Now we extend $c'$ to a colouring $c$ of $G$ to get a contradiction.
	There are two cases to consider.
	If there is a colour $i\in\{1,2,3,4\}$ such that $d_i(c',u)+d_i(c',v)\le 1$ we put $c(uv)=i$. Then $d_i(c,u) = 1$ and $d_i(c,v) \le 2$ or vice versa, so the edges of colour $i$ still form a collection of paths, and $c$ is as desired.
	Hence we can  assume that for every colour $i\in\{1,2,3,4\}$ we have $d_i(c',u)+d_i(c',v)=2$. Then, since $d_{G'}(u)+d_{G'}(v)\le 8$ we get that $d_0(c',u)=d_0(c',v)=0$. We simply put $c(uv)=0$.
\end{proof}

\begin{corollary}\label{lem:no1degree}
	For every $G \in \G$, graph $G$ does not contain $1$-vertices.
\end{corollary}
\begin{proof}
  Immediate from Lemma~\ref{lem:c1}, since $\Delta(G)\le 9$.
\end{proof}

\begin{lemma}\label{lem:c3}
		For every $G \in \G$, if graph $G$ does not contain \rxconf{c:kite}, then $G$ does not contain configuration~\rxconf{c:2degree2neighbours}.
\end{lemma}
\begin{proof}

	\begin{figure}[ht]
		\centering
		\subfloat[][]{
			\centering
			\begin{tikzpicture}[scale=0.95]
			\tikzstyle{whitenode}=[draw,circle,fill=white,minimum size=8pt,inner sep=0pt]
			\tikzstyle{blacknode}=[draw,circle,fill=black,minimum size=6pt,inner sep=0pt]
			\tikzstyle{tnode}=[draw,ellipse,fill=white,minimum size=8pt,inner sep=0pt]
			\tikzstyle{texte}=[fill=white, text=black]
			\draw (2,-3) node[whitenode] (v) [label=right:$v$] {}
			-- ++(60:1cm) node[whitenode] (y) [label=right:$y$] {\tiny{$2$}}
			-- ++(120:1cm) node[whitenode] (a) [] {}
			-- ++(240:1cm) node[whitenode] (x) [label=left:$x$] {\tiny{$2$}};
			\draw (x) edge []  node [label=left:] {} (v);
			\node[right=0pt] at (2.2,-1.2) {$a=b$};
			\end{tikzpicture}
			\label{fig:claim3case1}
		}
		\subfloat[][]{
			\centering
			\begin{tikzpicture}[scale=0.95]
			\tikzstyle{whitenode}=[draw,circle,fill=white,minimum size=8pt,inner sep=0pt]
			\tikzstyle{blacknode}=[draw,circle,fill=black,minimum size=6pt,inner sep=0pt]
			\tikzstyle{tnode}=[draw,ellipse,fill=white,minimum size=8pt,inner sep=0pt]
			\tikzstyle{texte}=[fill=white, text=black]
			\draw (0,-3) node[whitenode] (a) [label=below:$a$] {}
			-- ++(360:1cm) node[whitenode] (x) [label=below:$x$] {\tiny{$2$}}
			-- ++(360:1cm) node[whitenode] (v) [label=below:$v$] {}
			-- ++(360:1cm) node[whitenode] (y) [label=below:$y$] {\tiny{$2$}}
			-- ++(360:1cm) node[whitenode] (b) [label=below:$b$] {};
			\end{tikzpicture}
			\label{fig:claim3case2}
		}
		\caption{\label{fig:config3} Two possible cases when configuration \rxconf{c:2degree2neighbours} appears.}
	\end{figure}
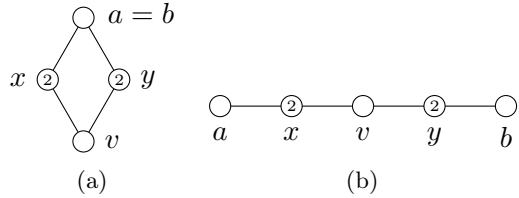

	Suppose for a contradiction that a vertex $v$ has two $2$-neighbours namely $x$ and $y$. Let $a$ (resp. $b$) be the neighbour of $x$ (resp. $y$) different than $v$. Note that $a$ may coincide with $b$ (see Fig.~\ref{fig:config3}).

	We consider two cases.

	\begin{enumerate}
		\item $a = b$

		Consider $G' = (G - x) \cup \{av\}$.  Note that $d_{G'}(v)=d_G(v) = 9$, as \rxconf{c:edge} is excluded in $G$. Moreover, $G'$ is simple, because if $av\in E(G)$, then $G$ contains \rxconf{c:kite} (triangle $(v, x, a)$ and edge $vy$). By the minimality of $G$ there exists a colouring $c'$ of $G'$. Now we define a colouring $c$ of $G$. For every edge $e\in E(G) \setminus \{ ax,\ ay,\ vx,\ vy\}$ we put $c(e) = c'(e)$.
		Let $\alpha$, $\beta$, $\gamma$ be the colours in $c'$ of $ay$, $yv$, $av$ respectively.
		We put $c(ay)=\alpha $, $c(yv)=\gamma$, $c(ax) = \gamma$ and $c(xv) = \beta$. Note that $d_i(c',v) = d_i(c, v)$ and $d_i(c', a) = d_i(c,a)$ for $i = 0, \ldots, 4$. Moreover, $d_i(c,x)$, $d_i(c,y)$ are not greater than $2$ for $i = 1,\ 2,\ 3,\ 4$ and $d_0(c,x)$, $d_0(c,y)$ are not greater than $1$, as otherwise $c'$ contains incident edges coloured with $0$.
		Hence, it suffices to show that there is no monochromatic cycle in $c$.

		Assume there is a monochromatic cycle $C$ in $c$.
		Then $C$ contains one of the edges of $ax$, $ay$, $bx$, $by$, because otherwise $C$ is monochromatic in $c'$, a contradiction.
		Since $d_G(x)=2$ and $d_G(y)=2$, it follows that $C$ contains the path $axv$ or $ayv$.
		By symmetry assume the former.
		It follows that $\beta=\gamma$ and $C$ is coloured by $\gamma$.
		Then $C$ contains $vy$ (because $vy$ is coloured by $\gamma$).
		Since $d_G(y)=2$, also $ya \in C$.
		Hence, $\alpha=\gamma$.
		Thus, we obtained that $\alpha=\beta=\gamma$ and the triangle $ayv$ was monochromatic in $c'$, a contradiction.

		\item $a \neq b$
		
        Since we excluded configuration \rxconf{c:edge}, $d(v) = 9$.
		Note that $G$ contains neither edge $av$ nor $vb$ as we excluded configuration \rxconf{c:kite}. Consider a simple graph $G' = G\setminus \{x,y\} \cup \{av,\ vb\} $. By the minimality of $G$ there exists a colouring $c'$ of $G'$. Now we define a colouring of $G$. For every edge  $e\in E(G) \setminus \{ ax,\ xv,\ vy,\ yb\}$ we put $c(e) = c'(e)$.
		Let $\alpha$ and $\beta$ be the colours of $av$ and $vb$ respectively.
		We can put $c(ax)=\alpha,\ c(xv)=\beta,\ c(vy)=\alpha,\ c(yb)=\beta$. Note that $d_i(c',v) = d_i(c,v),\ d_i(c', a) = d_i(c, a),\ d_i(c', b) = d_i(c, b)$ for $i = 0,\ldots, 4$. Moreover, $d_i(c,x),\ d_i(c,y)$ are not greater than $2$ for $ i = 1,\ 2,\ 3,\ 4$ and $d_0(c,x),\ d_0(c,y)$ are not greater than $1$, otherwise $c'$ contains incident edges coloured with $0$.

		We also claim that there is no monochromatic cycle in $c$. Indeed, if there is a cycle formed by edges in colour $\alpha$ (resp. $\beta$), then it goes through $x$ and $y$, so $\alpha = \beta$ and there is a monochromatic cycle in $c'$, a contradiction.
	\end{enumerate}

\end{proof}

\begin{lemma}\label{lem:config}
	For every $G \in \G$, graph $G$ does not contain any of configurations \rxconf{c:kite} and \rxconf{c:2consweak3neighbours} to \rxconf{c:2neighbour33bigface}.
\end{lemma}

\begin{proof}
This proof is done by a computer program, see Section~\ref{sec:computer} for details.
\end{proof}

\subsection{A minimal counterexample does not exist}
{\bf Initial charge}. We consider a planar embedding $\mathcal{M}=(V,E,F)$ of $G \in \G$ and attribute  a charge $\ch(v)=d(v)-4$ to each vertex $v$ and a charge $\ch(f)=\ell (f)-4$ to each face $f$.

\vspace{3mm}
\noindent {\bf Discharging rules} We introduce four discharging rules as follows:
\begin{itemize}
\item \rrule{r:degree2}{applies to a vertex $u$ with a neighbour $v$ of degree $2$. Then $u$ sends $1$ to $v$.}
\item \rrule{r:degree3}{applies to a vertex $u$ with a neighbour $v$ of degree $3$. Then $u$ sends $\tfrac13$ to $v$.}
\item \rrule{r:bigface}{applies to a face $f$ of length $5^+$ and its segment $(x,y,z)$. If $d(x)=d(z)=3$, then $f$ sends $\tfrac23$ to $y$ through $(x,y,z)$.
	If $d(x)=2$ or $d(z)=2$, then $f$ sends $\tfrac12$ to $y$ through $(x,y,z)$.}
\item \rrule{r:triangle}{applies to a vertex $u$ incident to a triangular face $f = (v,u,w)$. Then $u$ sends $m(d(v),d(u),d(w))$ through segment $(v, u, w)$ to $f$, as described in Table~\ref{table:triangle}.}
\end{itemize}

In \rxrule{r:bigface} we send a charge from a face to segments and then they pass the charge to vertices.
This may look overcomplicated at the moment, but will become handy in the proof.

\tabulinesep=1.5mm
\begin{table}
\centering
\begin{tabu}{c|c|c|c|c|c|c|c|}
\cline{2-8}
  & \multicolumn{1}{ c  }{\multirow{3}{*}{$a \leq 4$}} & \multicolumn{1}{ |c  }{\multirow{3}{*}{$a=5$}} & \multicolumn{1}{ |c | }{\multirow{3}{*}{$a =6$}} & \multicolumn{4}{c|}{$a \geq 7$}\\
  \cline{5-8}
 & \multicolumn{1}{ c  }{} & \multicolumn{1}{ |c  }{} & \multicolumn{1}{ |c | }{} & \multicolumn{1}{c|}{\multirow{2}{*}{$b\leq 4$}} & \multicolumn{2}{c|}{$b=5$} & \multicolumn{1}{ c|  }{\multirow{2}{*}{$b \geq 6$}} \\
 \cline{6-7}
 & \multicolumn{1}{ c  }{} & \multicolumn{1}{ |c  }{} & \multicolumn{1}{ |c  }{} & \multicolumn{1}{ |c|  }{} & $c=6$ & $c \geq 7$ & \multicolumn{1}{ c | }{} \\
 \hline
 \multicolumn{1}{ |c | }{$m(b,a,c)$} & $0$ & $\tfrac{1}{5}$ & $\tfrac13$ & $\tfrac12$ &  $\tfrac{7}{15}$ & $\tfrac25$ & $\tfrac13$ \\
 \hline
\end{tabu}
\caption{\label{table:triangle}A table to define the weight redistribution by Rule~\ref{r:triangle}. We consider $b \leq c$ and set $m(c,a,b)=m(b,a,c)$. Note that the case $b=c=5$ is not used in Rule~\ref{r:triangle}, because this would mean that edge $vw$ contradicts Lemma~\ref{lem:c1}.}
\end{table}


In what follows, we will show that for every face and for every vertex, the final charge, after applying the rules, is non-negative.

\begin{lemma}\label{l:facescharge}
  Let $G$ be a graph belonging to the family $\G$ and let $\mathcal{M}$ be its planar embedding. Then, after applying the discharging rules to $G$, for every face $f$ of $\mathcal{M}$ its final charge is non-negative.
\end{lemma}
\begin{proof}
   Let $f$ be a face of $\mathcal{M}$. Recall that by $\ell(f)$ we denote the length of $f$. We consider following cases.
\begin{itemize}
    \item ${\ell(f)=3}$

    The initial charge of $f$ is $-1$. It receives the charge from vertices on its boundary depending on their degree by \rxrule{r:triangle}. Consider a few subcases depending on the sorted degree sequence of the vertices on $f$ (note that these are the only possible triangles as we excluded \rxconf{c:edge}):
    \begin{itemize}
    \item $(4^-, 7^+, 7^+)$. The $7^+$ vertices send $\tfrac12$ to $f$, so its received charge is $2\cdot \tfrac12 = 1$.
    \item$(5,6, 8^-)$. These triangles are excluded by \rxconf{c:smalltriangles}.
    \item $(5,6,9)$. Then, $f$ receives $\tfrac15 +\tfrac13+\tfrac7{15} = 1$ from its vertices.
    \item $(5,7^+, 7^+)$. Then, $f$ receives $\tfrac15 + \tfrac25+\tfrac25 = 1$ from its vertices.
    \item $(6^+, 6^+, 6^+)$. Then, $f$ receives at least  $3\cdot \tfrac13 = 1$ from its vertices.
    \end{itemize}
     In each case $f$ receives at least $1$, so its final charge is non-negative.
    \item ${\ell(f)=4}$

    The initial and the final charge of $f$ is $0$, as it does not send or receive any charge.
    \item ${\ell(f)=5}$

    The initial charge of $f$ is $1$. Note that since \rxconf{c:edge}, \rxconf{c:2degree2neighbours} and $1$-vertices are excluded, $f$~has either two vertices of degree $3$ or at most one of degree $2$ on its boundary. In these situations $f$ sends $\tfrac23$ or at most $2\cdot \tfrac12$, respectively, by \rxrule{r:bigface}. In both cases its final charge remains non-negative.

    \begin{figure}[ht]
    \centering
    \subfloat[][]{
    \centering
    \begin{tikzpicture}[scale=0.95]
    \tikzstyle{whitenode}=[draw,circle,fill=white,minimum size=8pt,inner sep=0pt]
    \tikzstyle{blacknode}=[draw,circle,fill=black,minimum size=6pt,inner sep=0pt]
    \tikzstyle{tnode}=[draw,ellipse,fill=white,minimum size=8pt,inner sep=0pt]
    \tikzstyle{texte} =[fill=white, text=black]
         \draw (2,-3) node[whitenode] (v) [] {}
            -- ++(72:1cm) node[whitenode] (u) [] {\tiny{3}}
            -- ++(144:1cm) node[whitenode] (x1) [] {}
            -- ++(216:1cm) node[whitenode] (x2) [] {}
            -- ++(288:1cm) node[whitenode] (x3) [] {\tiny{3}};
        \draw (x3) edge []  node [label=left:] {} (v);
    \end{tikzpicture}
    \label{fig:c2casea}
    }
    \subfloat[][]{
    \centering
    \begin{tikzpicture}[scale=0.95]
    \tikzstyle{whitenode}=[draw,circle,fill=white,minimum size=8pt,inner sep=0pt]
    \tikzstyle{blacknode}=[draw,circle,fill=black,minimum size=6pt,inner sep=0pt]
    \tikzstyle{tnode}=[draw,ellipse,fill=white,minimum size=8pt,inner sep=0pt]
    \tikzstyle{texte} =[fill=white, text=black]
         \draw (2,-3) node[whitenode] (v) [] {}
            -- ++(72:1cm) node[whitenode] (u) [] {\tiny{$2$}}
            -- ++(144:1cm) node[whitenode] (x1) [] {}
            -- ++(216:1cm) node[whitenode] (x2) [] {}
            -- ++(288:1cm) node[whitenode] (x3) [] {};
        \draw (x3) edge []  node [label=left:] {} (v);
    \end{tikzpicture}
    \label{fig:c2caseb}
    }
    \caption{\label{fig:5face} $5$-faces sending any charge to vertices}
    \end{figure}
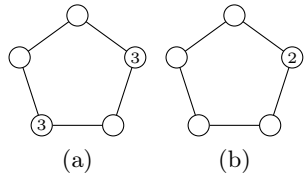
    \item ${\ell(f)\geq 6}$

    We will show the following claim.

  \begin{claim}\label{l:face6}
  	$f$ sends at most $\ell(f)/3$ of charge.
  \end{claim}
  \begin{proof}
  	\renewcommand\qedsymbol{$\lrcorner$}
    Let us implement \rxrule{r:bigface} in the following way.
	\begin{quote} \it
		Face $f$ sends $\tfrac13$ to every edge on its boundary.
	Next, for every edge $e = uv$ on the boundary of $f$ we do the following.
	Let $xuvy$ be a fragment of the facial walk of $f$.
	If $v$ is a $2$- or $3$-vertex, $e$ sends $\tfrac13$ to $u$ through $(x,u,v)$.
	Similarly, if $u$ is a $2$- or $3$-vertex, $e$ sends $\tfrac13$ to $v$ through $(u,v,y)$.
	Otherwise, $e$ sends $\tfrac16$ both to $u$ and $v$ through $(x,u,v)$ and $(u,v,y)$, respectively.
	\end{quote}
	Note that every edge sends at most $\tfrac13$ of charge, as $3^-$-vertices cannot be connected by an edge, because \rxconf{c:edge} is excluded.
	Hence, it suffices to show that for every segment $s=(u,v,w)$, if $v$ fulfills the conditions from \rxrule{r:bigface}, then $v$ gets appropriate charge from $s$.
	Indeed, if $d(u)=d(w)=3$, then $v$ gets $\tfrac13$ through $s$ from both $uv$ and $vw$, so it gets $\tfrac23$ as required.
	Otherwise, $d(u)=2$ or $d(w)=2$, by symmetry assume the latter.
	Then, $v$ gets $\tfrac13$ through $s$ from $vw$ and at least $\tfrac16$ from $vu$ through $s$, so $v$ gets at least $\tfrac12$ through $s$, as required.
  \end{proof}

   Since the initial charge of $f$ is $\ell(f)-2$ and $\ell(f)-2-\ell(f)/3\ge 0$ whenever   ${\ell(f)\geq 6}$, this concludes the proof of Lemma~\ref{l:facescharge}.
\end{itemize}
\end{proof}

\begin{lemma}\label{l:vertexcharge}
  Let $G$ be a graph belonging to the family $\G$. Then, after applying the discharging rules to $G$, the final charge of every vertex $v$ is non-negative.
\end{lemma}
\begin{proof}
   Let $v$ be a vertex of $G$. Let us state simple observations which follow from the discharging rules, and will be used frequently in the remainder.
   \begin{obs}
     Vertex $v$ sends nothing to incident non-triangular segments.
   \end{obs}
   \begin{obs}
     Vertex $v$ sends at most $\tfrac12$ to an incident triangular segment.
   \end{obs}
   \begin{obs}
     Vertex $v$ sends at most $\tfrac13$ to an incident triangular segment with no $5^-$ vertices.
   \end{obs}

   In what follows we consider cases depending on the degree of $v$.
    \\\\ Assume ${d(v) \leq 1}$.
       By Corollary \ref{lem:no1degree} there are no vertices of this kind in $G$.
    \\\\ Assume ${d(v) = 2}$.
      The initial charge of $v$ is $-2$. By \rxrule{r:degree2} $v$ receives $1$ from each of the two neighbours, and since it does not send any charge, its final charge is $0$.
    \\\\ Assume ${d(v) = 3}$.
      The initial charge of $v$ is $-1$. By \rxrule{r:degree3} $v$ receives $\tfrac13$ from each of the three neighbours, and since it does not send any charge, its final charge is $0$.
    \\\\ Assume ${d(v) = 4}$.
      The initial charge and the final charge of $v$ are equal to $0$, as it does not send or receive any charge.
    \\\\ Assume ${d(v) = 5}$.
      The initial charge of $v$ is $1$. Vertex $v$ sends charge only to incident triangular segments by \rxrule{r:triangle}. Thus, each incident segment to $v$ receives at most $\tfrac15$ from $v$. Hence, $v$ sends at most $5\cdot \tfrac15 = 1$ and its final charge is non-negative.
    \\\\ Assume ${d(v) = 6}$.
      The initial charge of $v$ is $2$. Vertex $v$ sends charge only to incident triangular segments by \rxrule{r:triangle}. Thus, each segment incident to $v$ receives at most $\tfrac13$ from $v$. Hence, $v$ sends at most $6\cdot \tfrac13 = 2$ and its final charge is non-negative.
    \\\\ Assume ${d(v) = 7}$.
      The initial charge of $v$ is $3$. Vertex $v$ sends charge only to incident triangular segments by \rxrule{r:triangle}.
       We claim that $v$ is incident to at most two triangular segments such that each contains a $4$-vertex. Indeed, if $v$ is incident to a triangular segment containing a $4$-vertex, $v$ cannot have another $4$-neighbour as the configuration \rxconf{c:kite} is excluded. Hence every such segment contains edge $vx$, and there are at most two segments with this property.

      By \rxrule{r:triangle} $v$ sends $\tfrac12$ to triangular segments containing a $4$-vertex and at most $\tfrac25$ to each of the remaining segments (by \rxconf{c:smalltriangles} the triangular segments $(7,6,5)$ are excluded). Thus, $v$ sends at most $2\cdot \tfrac12 + (7-2)\tfrac25  = 3$ to the segments, so its final charge is non-negative.
    \\\\ Assume ${d(v) = 8}$.
        The initial charge of $v$ is $4$. By \rxconf{c:edge} $v$ cannot have $2$-neighbours, so $v$ passes the charge only to  triangular segments and $3$-vertices by \rxrule{r:bigface} and \rxrule{r:triangle}. Let us consider the following cases:
\begin{enumerate}
    \item There is a weak $3$-neighbour $a$.
    \\ By \rxconf{c:kite} and \rxconf{c:8weak3degree4} $v$ has no $4^-$-neighbours apart from $a$. This observation and \rxconf{c:smalltriangles} imply that $v$ sends $\tfrac13$ to $a$, $\tfrac12$ to each of the two incident triangular segments containing $a$ and at most $\tfrac25$ to each of the remaining segments. Thus, $v$ sends at most $$\tfrac13+2\cdot \tfrac12 + \tfrac25(8-2) = 3\tfrac{11}{15} \leq 4.$$

    \item There is a semi-weak $3$-neighbour $a$.
    \\ By \rxconf{c:kite} $v$ has no more $3$-neighbours. Since $a$ is the semi-weak neighbour, $v$ is incident to at most $7$ triangular segments. Thus, $v$ sends at most $$\tfrac13+\tfrac12(8-1) = 3\tfrac56 \leq 4.$$

    \item There are no weak or semi-weak $3$-neighbours.
    \\ Let $k$ be the number of $3$-neighbours of $v$. For each such $3$-neighbour $a$ there are exactly two non-triangular $v$-segments containing $a$. Since every non-triangular $v$-segment contains at most two $3$-neighbours of $v$, there are at least $k$ non-triangular $v$-segments, so $v$ sends at most:
    $$\tfrac13k + \tfrac12(8-k) = 4-\tfrac16k \leq 4.$$

\end{enumerate}
In each case $v$ sends at most $4$ to adjacent vertices and segments. Thus, its final charge remains non-negative.

    \\\\ Finally, assume ${d(v) = 9}$.
        The initial charge of $v$ is $5$. We consider several cases:

    \noindent {\bf Case 1} $v$ has no $2$-neighbours.
    \\Let $n_{3_{1,0}}$ be the number of $3$-neighbours of $v$ which are not weak and let $\S$ be the set of non-triangular $v$-segments. Recall that no pair of $3$-neighbours of $v$ can be adjacent as \rxconf{c:edge} is excluded.
    We claim that
    \begin{equation}\label{eq:facesbound}
      |\S| \geq \ceil{\tfrac{n_{3_{1,0}}}{2}}.
    \end{equation}
    Indeed, every non-weak $3$-neighbour belongs to at least one non-triangular $v$-segment and each non-triangular $v$-segment contains at most two such $3$-neighbours.
    \\ Let $n_{3_2}$ be the number of weak $3$-neighbours of $v$. By \rxconf{c:2weak3neighbours3neighbour} we have a bound $n_{3_2} \leq 2$.
    We will consider three subcases depending on $n_{3_2}$.
    \vspace{3mm}

    \noindent {\bf Case 1.1} $n_{3_2} = 0$
    \\ First, assume $n_{3_{1,0}} = 8$. Then, $v$ has at most one $8^+$-neighbour, so since \rxconf{c:edge} is excluded, $v$ is incident to at most two triangular segments. In total $v$ sends at most
    $$8\cdot \tfrac13 + 2\cdot \tfrac12 = 3\tfrac23,$$
    which is smaller than the initial charge.
    Thus, assume $n_{3_{1,0}} \neq 8$. Then $v$ sends at most $\tfrac13n_{3_{1,0}} + \tfrac12(9 - |\S|)$ to adjacent vertices and segments, so its final charge is at least
    $$ 5 - \tfrac13n_{3_{1,0}} - \tfrac12(9 - |\S|) = \tfrac12 + \tfrac12|\S| - \tfrac13n_{3_{1,0}} \stackrel{\eqref{eq:facesbound}}{\geq} \tfrac12 + \tfrac12 \ceil{\tfrac{n_{3_{1,0}}}{2}} - \tfrac13n_{3_{1,0}}.$$
    For odd $n_{3_{1,0}}$ we get $\tfrac12 + \tfrac12 \tfrac{n_{3_{1,0}} + 1}2 - \tfrac13n_{3_{1,0}} = \tfrac34 - \tfrac{n_{3_{1,0}}}{12}$, which is non-negative when $n_{3_{1,0}} \leq 9$.
    For even $n_{3_{1,0}}$, we get $\tfrac12 + \tfrac{n_{3_{1,0}} }{4} - \tfrac13n_{3_{1,0}} = \tfrac12 - \tfrac{n_{3_{1,0}}}{12}$.
    Since $n_{3_{1,0}} \neq 8$ we are left with even $n_{3_{1,0}} \leq 6$ and then $\tfrac12 - \tfrac{n_{3_{1,0}}}{12} \geq 0$, as required.
    \vspace{3mm}

    \noindent {\bf Case 1.2} $n_{3_2} = 1$
    \\ Note that $n_{3_{1,0}} \leq 6$. Indeed, $v$ has nine neighbours, one of them is a weak $3$-neighbour and by \rxconf{c:edge} the neighbours it shares with $v$ have degree $8^+$.

      \begin{claim}\label{l:four3neighbours}
              If $n_{3_{1,0}} = 4$  the final charge of $v$ is non-negative.
      \end{claim}
      \begin{proof}

        \begin{figure}[ht]
                                  \centering
                                  \subfloat[][]{
                                  \centering
                                  \begin{tikzpicture}[scale=0.95]
                                  \tikzstyle{whitenode}=[draw,circle,fill=white,minimum size=8pt,inner sep=0pt]
                                  \tikzstyle{blacknode}=[draw,circle,fill=black,minimum size=6pt,inner sep=0pt]
                                  \tikzstyle{smalldot}=[draw,circle,fill=black,minimum size=1pt,inner sep=0pt]
                                  \tikzstyle{tnode}=[draw,ellipse,fill=white,minimum size=8pt,inner sep=0pt]
                                  \tikzstyle{texte}=[fill=white, text=black]
                                    \draw (2,-3) node[whitenode] (v) [] {\tiny{$v$}}
                                       -- ++(70:1cm) node[whitenode] (a) [] {}
                                       -- ++(120:0.75cm) node[whitenode] (b) [] {\tiny{$3$}}
                                       -- ++(240:0.75cm) node[whitenode] (c) [] {};
                                    \draw (1.1, -2.8) node[whitenode] (x) [] {\tiny{$3$}}
                                      -- ++ (205:0.5cm) node[whitenode] (y) [] {};

                                    \draw (1.37, -3.6) node[whitenode] (z) [] {\tiny{$3$}}
                                      -- ++ (190:0.5cm) node[whitenode] (z1) [] {};

                                    \draw (0.68, -3.25) node[smalldot] (d1) []{};
                                    \draw (0.7025, -3.36) node[smalldot] (d2) []{};
                                    \draw (0.725, -3.47) node[smalldot] (d3) []{};

                                  \draw (3, -3) node[whitenode] (k) [] {\tiny{$3$}}
                                      -- ++(330:0.5cm) node[whitenode] (l) [] {};

                                  \draw (2.7, -3.8) node[whitenode] (m) [] {\tiny{$3$}}
                                     -- ++(350:0.5cm) node[whitenode] (m1) [] {};

                                  \draw (3.34, -3.475) node[smalldot] (d5) []{};
                                  \draw (3.30, -3.575) node[smalldot] (d4) []{};
                                  \draw (3.26, -3.675) node[smalldot] (d6) []{};

                                   \draw (x) edge []  node [label=left:] {} (v);
                                   \draw (a) edge []  node [label=left:] {} (v);
                                   \draw (m) edge []  node [label=left:] {} (v);
                                   \draw (k) edge []  node [label=left:] {} (v);
                                   \draw (z) edge []  node [label=left:] {} (v);
                                   \draw (b) edge []  node [label=left:] {} (v);
                                   \draw (c) edge []  node [label=left:] {} (v);

                                   \draw (c) edge []  node [label=left:] {} (x);
                                   \draw (a) edge []  node [label=left:] {} (k);

                                   \draw (1.7, -4) node[whitenode] (s) [] {};
                                   \draw (2.3, -4) node[whitenode] (t) [] {};

                                   \draw (s) edge []  node [label=left:] {} (v);
                                   \draw (s) edge []  node [label=left:] {} (z);
                                   \draw (t) edge []  node [label=left:] {} (v);
                                   \draw (t) edge []  node [label=left:] {} (m);
                                   \draw (t) edge []  node [label=left:] {} (s);

                                  \end{tikzpicture}
                                  \label{fig:deg9claim1case1}
                                  }
                                  \subfloat[][]{
                                  \centering
                                  \begin{tikzpicture}[scale=0.95]
                                  \tikzstyle{whitenode}=[draw,circle,fill=white,minimum size=8pt,inner sep=0pt]
                                  \tikzstyle{blacknode}=[draw,circle,fill=black,minimum size=6pt,inner sep=0pt]
                                  \tikzstyle{tnode}=[draw,ellipse,fill=white,minimum size=8pt,inner sep=0pt]
                                  \tikzstyle{smalldot}=[draw,circle,fill=black,minimum size=1pt,inner sep=0pt]
                                  \tikzstyle{texte}=[fill=white, text=black]
                                  \draw (2,-3) node[whitenode] (v) [] {\tiny{$v$}}
                                     -- ++(70:1cm) node[whitenode] (a) [] {}
                                     -- ++(120:0.75cm) node[whitenode] (b) [] {\tiny{$3$}}
                                     -- ++(240:0.75cm) node[whitenode] (c) [] {};

                                 \draw (1.1, -2.8) node[whitenode] (x) [] {\tiny{$3$}}
                                       -- ++ (205:0.5cm) node[whitenode] (y) [] {};

                                 \draw (1.37, -3.6) node[whitenode] (z) [] {\tiny{$3$}}
                                       -- ++ (190:0.5cm) node[whitenode] (z1) [] {};

                                 \draw (0.68, -3.25) node[smalldot] (d1) []{};
                                 \draw (0.7025, -3.36) node[smalldot] (d2) []{};
                                 \draw (0.725, -3.47) node[smalldot] (d3) []{};

                                 \draw (3, -3) node[whitenode] (k) [] {\tiny{$3$}}
                                     -- ++(330:0.5cm) node[whitenode] (l) [] {};

                                 \draw (2.7, -3.8) node[whitenode] (m) [] {\tiny{$3$}}
                                    -- ++(350:0.5cm) node[whitenode] (m1) [] {};

                                 \draw (3.34, -3.475) node[smalldot] (d5) []{};
                                 \draw (3.30, -3.575) node[smalldot] (d4) []{};
                                 \draw (3.26, -3.675) node[smalldot] (d6) []{};

                                 \draw (x) edge []  node [label=left:] {} (v);
                                 \draw (a) edge []  node [label=left:] {} (v);
                                 \draw (m) edge []  node [label=left:] {} (v);
                                 \draw (k) edge []  node [label=left:] {} (v);
                                 \draw (z) edge []  node [label=left:] {} (v);
                                 \draw (b) edge []  node [label=left:] {} (v);
                                 \draw (c) edge []  node [label=left:] {} (v);

                                 \draw (a) edge []  node [label=left:] {} (k);

                                 \draw (2, -4) node[whitenode] (s) [] {};
                                 \draw (s) edge []  node [label=left:] {} (v);
                                 \draw (s) edge []  node [label=left:] {} (z);
                                 \draw (s) edge []  node [label=left:] {} (m);

                                 \draw (1.12, -2.33) node[whitenode] (t) [] {};
                                 \draw (t) edge []  node [label=left:] {} (c);
                                 \draw (t) edge []  node [label=left:] {} (x);
                                 \draw (t) edge []  node [label=left:] {} (v);
                                  \end{tikzpicture}
                                  \label{fig:deg9claim1case2}
                                  }
                                  \caption{\label{fig:deg9claim1} Possible configurations of neighbours of $v$ \\for $|\S| = 2$, $n_{3_2} = 1$ and $n_{3_{1,0}} = 4$}
        \end{figure}
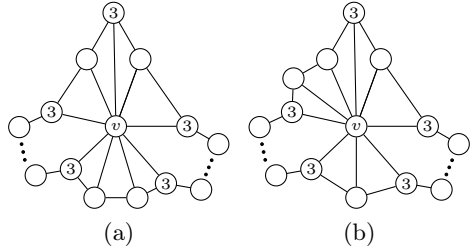

        Note that $v$ has exactly five $3$-neighbours. Recall that $|\S| \geq \ceil{\tfrac{n_{3_{1,0}}}{2}}$. Hence, $|\S|\geq 2$. Suppose $|\S| =2$. Since $n_{3_{1,0}} = 4$ and each of these four $3$-neighbours belongs to at least one non-triangular $v$-segment, each of the two non-triangular segments contains two semi-weak $3$-neighbours of $v$. Hence, we consider two cases --- see Figures \ref{fig:deg9claim1case1} and
        \ref{fig:deg9claim1case2}, as they are the only possible configurations. In both cases, by \rxconf{c:edge}, $v$ belongs to a segment with the degree sequence $(8^+,9,8^+)$, which gets only $\tfrac13$ of charge from $v$.

        Then $v$ sends at most
        $$6\cdot\tfrac12 + \tfrac13 + 5\cdot\tfrac13 = 5.$$

        Finally, suppose $|\S|\geq3$. Then, $v$ sends at most $$\tfrac12(9 - |\S|) + 5\cdot \tfrac13 = 6\tfrac16 - \tfrac12|\S| \leq 6\tfrac16 - \tfrac12\cdot 3 = 4\tfrac23 \leq 5.$$
      \end{proof}

      \begin{claim}\label{l:six3neighbours}
          If $n_{3_{1,0}}  = 6$ the final charge of $v$ is non-negative.
      \end{claim}
      \begin{proof}
        For $n_{3_{1,0}} = 6$ the neighbourhood of $v$ consists of a weak $3$-neighbour, its two $8^+$-neighbours and six $3$-neighbours which are not weak. In particular, the six non-weak $3$-neighbours are consecutive around $v$. By \rxconf{c:edge}, it follows that for each pair $x$, $y$ of these six neighbours the segment $(x, v, y)$ is non-triangular. Hence, $|\S| \geq 5$.

        Then, $v$ sends at most
        $$ 7\cdot \tfrac13 + (9 - |S|)\cdot \tfrac 12 = 4\tfrac13 < 5.$$

      \end{proof}
      Now we will show that the final charge of $v$ is non-negative also in the remaining cases.
      Vertex $v$ sends at most  $\tfrac13 + \tfrac13n_{3_{1,0}} + \tfrac12(9 - |\S|)$ to adjacent vertices and segments, so its final charge is at least
      $$5 - \tfrac13 - \tfrac13n_{3_{1,0}} - \tfrac12(9 - |\S|)= \tfrac16 + \tfrac12|\S| -  \tfrac13n_{3_{1,0}} \stackrel{\eqref{eq:facesbound}}{\geq} \tfrac16 +\tfrac12\ceil{\tfrac{n_{3_{1,0}}}2} - \tfrac13n_{3_{1,0}}.$$

      For odd $n_{3_{1,0}}$, this value is equal to $\tfrac16 + \tfrac12\tfrac{n_{3_{1,0}} + 1}{2} - \tfrac13n_{3_{1,0}} = \tfrac5{12} - \tfrac{n_{3_{1,0}}}{12}$
      which is non-negative for $n_{3_{1,0}}\leq 5$.
      For even $n_{3_{1,0}}$ it is equal $\tfrac16 + \tfrac12\tfrac{n_{3_{1,0}}}{2} - \tfrac13n_{3_{1,0}} = \tfrac16 - \tfrac{n_{3_{1,0}}}{12}$, which is non-negative for $n_{3_{1,0}}\leq 2$.
      Since $n_{3_{1,0}} \leq 6$, these are all the cases not covered by Claims \ref{l:four3neighbours} and
      \ref{l:six3neighbours}.
      \vspace{3mm}

      \noindent {\bf Case 1.3} $n_{3_2} = 2$
      \\Let $a$, $b$ be the weak $3$-neighbours of $v$.
      \begin{figure}[!ht]
                \centering
                \subfloat[][]{

                    \begin{tikzpicture}[scale=0.95]
                    \tikzstyle{whitenode}=[draw,circle,fill=white,minimum size=8pt,inner sep=0pt]
                    \tikzstyle{blacknode}=[draw,circle,fill=black,minimum size=6pt,inner sep=0pt]
                    \tikzstyle{whitenode}=[draw,circle,fill=white,minimum size=8pt,inner sep=0pt]
                    \tikzstyle{blacknode}=[draw,circle,fill=black,minimum size=6pt,inner sep=0pt]
                    \tikzstyle{tnode}=[draw,ellipse,fill=white,minimum size=8pt,inner sep=0pt]
                    \tikzstyle{texte} =[fill=white, text=black]
                    \draw (2,-3) node[whitenode] (v) [label=left:$v$] {}
                        -- ++(60:1.04cm) node[whitenode] (x1) [] {}
                        -- ++(150:0.6cm) node[whitenode] (v1) [] {\tiny{$3$}}
                        -- ++(210:0.6cm) node[whitenode] (x2) [] {}
                        -- ++(300:1.04cm) node[whitenode] (u2) [] {}
                        -- ++(240:1.04cm) node[whitenode] (x3) [] {}
                        -- ++(330:0.6cm) node[whitenode] (v2) [] {\tiny{$3$}}
                        -- ++(30:0.6cm) node[whitenode] (x4) [] {};

                    \draw (3,-3) node[whitenode] (v3) [] {};
                    \draw (1.13,-3.5) node[whitenode] (v4) [] {};
                    \draw (1.13,-2.5) node[whitenode] (v5) [] {};

                    \draw (v1) edge []  node [] {} (v);
                    \draw (v2) edge []  node [] {} (v);
                    \draw [dashed] (v3) edge []  node [] {} (v);
                    \draw [dashed] (v4) edge []  node [] {} (v);
                    \draw [dashed] (v5) edge []  node [] {} (v);
                    \draw (x4) edge []  node [] {} (v);
                    \draw (x4) edge []  node [] {} (v2);
                    \draw (x1) edge []  node [] {} (v1);
                    \draw (x2) edge []  node [] {} (v1);

                    \end{tikzpicture}
                    \label{fig:two3neighboursCasea}
                    }
                    \subfloat[][]{
                    \begin{tikzpicture}[scale=0.95]
                    \tikzstyle{whitenode}=[draw,circle,fill=white,minimum size=8pt,inner sep=0pt]
                    \tikzstyle{blacknode}=[draw,circle,fill=black,minimum size=6pt,inner sep=0pt]
                    \tikzstyle{whitenode}=[draw,circle,fill=white,minimum size=8pt,inner sep=0pt]
                    \tikzstyle{blacknode}=[draw,circle,fill=black,minimum size=6pt,inner sep=0pt]
                    \tikzstyle{tnode}=[draw,ellipse,fill=white,minimum size=8pt,inner sep=0pt]
                    \tikzstyle{texte} =[fill=white, text=black]
                    \draw (2,-3) node[whitenode] (v) [label=right:$v$] {}
                        -- ++(60:1.04cm) node[whitenode] (x1) [] {}
                        -- ++(150:0.6cm) node[whitenode] (v1) [] {\tiny{$3$}}
                        -- ++(210:0.6cm) node[whitenode] (x2) [] {}
                        -- ++(300:1.04cm) node[whitenode] (u2) [] {}
                        -- ++(240:1.04cm) node[whitenode] (x3) [] {}
                        -- ++(330:0.6cm) node[whitenode] (v2) [] {\tiny{$3$}}
                        -- ++(30:0.6cm) node[whitenode] (x4) [] {};

                    \draw (1,-3) node[whitenode] (v3) [] {};
                    \draw (1.13,-3.5) node[whitenode] (v4) [] {};
                    \draw (1.13,-2.5) node[whitenode] (v5) [] {};

                    \draw (v1) edge []  node [] {} (v);
                    \draw (v2) edge []  node [] {} (v);
                    \draw [dashed] (v3) edge []  node [] {} (v);
                    \draw [dashed] (v4) edge []  node [] {} (v);
                    \draw [dashed] (v5) edge []  node [] {} (v);
                    \draw (x4) edge []  node [] {} (v);
                    \draw (x4) edge []  node [] {} (v2);
                    \draw (x1) edge []  node [] {} (v1);
                    \draw (x2) edge []  node [] {} (v1);

                    \end{tikzpicture}
                    \label{fig:two3neighboursCaseb}
                    }
                    \caption{\label{fig:9vertexno2neighbourtwo3neigbours} Possible configurations of vertices incident to $v$}
      \end{figure}
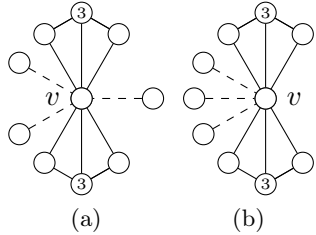

      \begin{claim}\label{l:bigfaceorcheaptriangle}
      In this case $v$ is incident to at least one triangular segment not containing a $5^-$ vertex or to at least one non-triangular segment.
      \end{claim}

      \begin{proof}
        By \rxconf{c:2consweak3neighbours} we know that the two weak $3$-neighbours $a$ and $b$ do not share a neighbour other than $v$, so there are exactly three neighbours of $v$ which are neither $a$ nor $b$ nor the neighbours of $a,\ b$. Let $C$ be the set of such vertices. There are two ways of how the vertices of $C$ may be arranged around $v$: either exactly two of them are consecutive (see Fig \ref{fig:two3neighboursCasea}) or
        all of them are consecutive (see Fig \ref{fig:two3neighboursCaseb}).

        Suppose that all segments incident to $v$ are triangular, for otherwise we are done. In the case $(a)$, the two consecutive vertices of $C$ are neighbours, so by \rxconf{c:edge} one of them, call it $x$, has degree at least $6$. Taking $x$, $v$ and the common neighbour of $v$, $x$, and $a$ or $b$ we get the triangular segment from the claim.

        In the case $(b)$, the sought triangular segment is formed by $v$ and neighbours of $a$ and~$b$.
      \end{proof}

      According to \rxrule{r:triangle}, $v$ sends $\tfrac13$ to every triangular segment containing no vertices of degree $5^-$, if any. Moreover, $v$ sends no charge to non-triangular segments. Hence, by Claim \ref{l:bigfaceorcheaptriangle} $v$ sends at most $\tfrac13 + 8 \cdot \tfrac12$ to triangular segments. Also, $v$ sends $2\cdot \tfrac13$ to $a$ and $b$, and nothing more to the other neighbours by \rxconf{c:edge} and
      \rxconf{c:2weak3neighbours3neighbour}. Thus, the total charge sent by $v$ is at most
      $$2\cdot\tfrac13 + \tfrac13 + 8 \cdot \tfrac12 = 5.$$

\noindent {\bf Case 2} $v$ has a $2$-neighbour.
\\By \rxconf{c:2degree2neighbours} $v$ has exactly one $2$-neighbour, let us denote it by $a$. By \rxconf{c:weak3degree2} $v$ does not have weak $3$-neighbours. In the next paragraphs we study the dependence between the number of $3$-neighbours of $v$ and that of non-triangular $v$-segments.

  We introduce the following notation. Let $\mathcal{N}_3$ be the set of $3$-neighbours of $v$. For a segment $s$ let $n_3(s)$ be the number of $3$-vertices in $s$. Let $\S$ be the set of all non-triangular $v$-segments and let $\S^i$ be the set of segments of length $i$ incident with $v$ (recall that the length of a segment is defined as the length of face it belongs to). Moreover, we define
  \begin{itemize}
    \item $\S_d = \{s\in\S\ |\ s \mbox{ contains a } d \mbox{-neighbour of } v\}$,
    \item $\S_{d, d'} = \{(x,v,y)\in\S\ |\ \{d_G(x), d_G(y)\} = \{d,d'\}\}$
  \end{itemize}
  For example, $\S_{3}$ is a set of segments containing a $3$-neighbour of $v$ (note that it may contain a $2$-neighbour or another $3$-neighbour).
  We extend this notation in a natural way to degree lower bounds, for example, $\S_{2, 4^+}$ is a set of segments containing exactly one $2$-neighbour of $v$ and a $4^+$-neighbour of $v$.

  We also define $\S_l^d = \S^d \cap \S_l$, e.g., $\S_{3,4+}^{5+} = \S^{5+}\cap \S_{3,4+}$. Moreover, let $$\mathcal{N}_3^* = \{x\in \mathcal{N}_3\ |\ \mbox{if } x \text{ belongs to } s\in\S, \mbox{ then } s\in\S_3^4 \cup \S_{3,4+}^{5+} \}.$$
  We also use the notation defined above with respect to any subgraph $G'$ of $G$ and a vertex $v'\in V(G')$, for example $\S(G',v')$ is the set of all non-triangular $v'$-segments, $\S_d(G', v') = \{ s\in \S(G', v')\ |\ s \text{ contains a } d \text{-neighbour of } v'\}$ etc.

  \begin{claim}\label{l:faces3degree}
    Let $G'$ be a subgraph of $G$, and let $v'$ be a vertex of $G'$ such that $v'$ has no weak $3$-neighbours and it is not incident with a $4$-face with two semi-weak $3$-neighbours on its boundary.
    Then $$|\mathcal{N}_3^*(G',v')|\leq \tfrac32 | \S_3^4(G',v') \cup \S_{3,4+}^{5+}(G',v')|.$$
  \end{claim}
  \begin{proof}
    We use a discharging argument. Each segment $s\in\S_3^4(G',v') \cup \S_{3,4+}^{5+}(G',v')$ sends the charge using the following rules:
    \vspace{3mm}

    {\it Consider a segment $s=(x,v',y)\in\S_3^4(G',v') \cup \S_{3,4+}^{5+}(G',v')$.
    First assume $s$ contains a semi-weak $3$-neighbour, say $x$.
    By our assumption $y$ is not semi-weak $3$-neighbour.
    Then $s$ sends $1$ to $x$ and $\tfrac12$ to $y$.
    Otherwise, suppose that in our plane embedding $x$ appears just before $y$ among the neighbours of $v$ in the clockwise order  around $v$.
    Then, $s$ sends $1$ to $y$ and $\tfrac12$ to $x$.}

    \vspace{3mm}

    Observe that to show the claim it suffices to prove that ($i$) every vertex $x\in\mathcal{N}_3^*(G',v')$ gets at least $1$ charge and ($ii$) every segment of $\S_3^4(G',v') \cup \S_{3,4+}^{5+}(G',v')$ sends at most $\tfrac32$ charge.

    Consider $x\in \mathcal{N}_3^*(G',v')$. If $x$ is incident to two segments from $\S(G',v')$, then $x$ gets at least $\tfrac12$ from each of them, so at least $1$ charge in total. Otherwise, $x$ is incident to exactly one segment from $\S(G',v')$, because weak $3$-neighbours are excluded.
    Hence, $x$ is a semi-weak $3$-neighbour and it gets $1$ charge from the incident segment, so ($i$) holds. For ($ii$) it suffices to show that each segment from $\S_3^4(G',v')\cup \S_{3, 4^+}^{5+}(G',v')$  has on its boundary at most one semi-weak
    $3$-neighbour of $v$. Indeed, a segment from $\S_{3, 4^+}^{5+}(G',v')$ has exactly one by definition and a segment of length $4$ cannot be incident to two semi-weak $3$-neighbours by our assumption.
  \end{proof}

  \begin{claim}\label{l:bigface}
    If $|\S_{2, 4^+}^4 \cup \S_2^{5+}\cup \S_{3,3}^{5+} \cup \S_{4^+, 4^+}| \geq 1$, then $v$ has non-negative final charge.
  \end{claim}

  \begin{proof} Note that
    $|\S|= |\S_{2,4^+}^4| + |\S_2^{5+}| + |\S_{3,3}^{5+}| +
    |\S_3^4\cup \S_{3,4+}^{5+}| + |\S_{4^+, 4^+}|$.
    Consider a $3$-neighbour $x\in \mathcal{N}_3\setminus \mathcal{N}_3^*$. By \rxconf{c:weak3degree2} $x$ is not a weak $3$-neighbour, so it is incident to at least one segment $s\in\S$. Since $x\in \mathcal{N}_3\setminus \mathcal{N}_3^*$, we may choose $s$ so that $s\in \S_3\setminus(\S_3^4\cup \ \S_{3,4+}^{5+}) \subseteq \S_2^{5+}\cup \S_{3,3}^{5+}.$
    There are at most $|\S_2^{5+}|$, resp. $2|\S_{3,3}^{5+}|$, such $x$'s incident to a segment from $\S_2^{5+}$, resp. $\S_{3,3}^{5+}$. Hence, $|\mathcal{N}_3 \setminus \mathcal{N}_3^*| \leq |\S_2^{5+}| + 2 |\S_{3,3}^{5+} |$.
    By Claim \ref{l:faces3degree} (applied with $G'=G$ and $v'=v$, note that the assumptions are met because \rxconf{c:weak3degree2} and \rxconf{c:2neighbour33bigface} are excluded), it follows that $|\mathcal{N}_3| \leq |\S_2^{5+}| + 2|\S_{3,3}^{5+}| + \tfrac32|\S_3^4\cup \S_{3,4+}^{5+}|$.

    Thus, $v$ sends at most
  \begin{equation}
    \begin{aligned}
          &  1 + \tfrac13|\mathcal{N}_3| + \tfrac12(9-|\S|) = 5\tfrac12 + \tfrac13|\mathcal{N}_3| -\tfrac12|\S| \leq \\
          & \leq 5\tfrac12 + \tfrac13(|\S_2^{5+}| + 2|\S_{3,3}^{5+}| + \tfrac32|\S_3^4\cup \S_{3,4+}^{5+}|) - \tfrac12(|\S_{2,4^+}^4| + |\S_2^{5+}| + |\S_{3,3}^{5+}| +
          |\S_3^4\cup \S_{3,4+}^{5+}| + |\S_{4^+, 4^+}|) = \\
         & = 5\tfrac12 - \tfrac16|\S_2^{5+}| + \tfrac16|\S_{3,3}^{5+}|-\tfrac12|\S_{4^+, 4^+}| - \tfrac12|\S_{2,4^+}^4|.
   \end{aligned}
  \end{equation}

  The initial charge increased by the charge $v$ receives from segments by \rxrule{r:bigface} is equal to $$5 + \tfrac12|\S_2^{5+}| + \tfrac23|\S_{3,3}^{5+}|.$$

  Hence, the final charge at $v$ is at least
  \begin{equation}
    \begin{aligned}
    &\left(5 + \tfrac12|\S_2^{5+}| + \tfrac23|\S_{3,3}^{5+}|\right) - \left( 5\tfrac12 -\tfrac16|\S_2^{5+}| + \tfrac16|\S_{3,3}^{5+}|-\tfrac12|\S_{4^+, 4^+}| - \tfrac12|\S_{2,4^+}^4|  \right) = \\
    &\tfrac23|\S_2^{5+}| + \tfrac12| \S_{3,3}^{5+}| + \tfrac12|\S_{4^+, 4^+}| + \tfrac12|\S_{2,4^+}^4| -\tfrac12
  \end{aligned}
  \end{equation}
  which is non-negative if $|\S_{2, 4^+}^4 \cup \S_2^{5+}\cup \S_{3,3}^{5+} \cup \S_{4^+, 4^+}| \geq 1$.
\end{proof}

   As we excluded $\rxconf{c:weak3degree2}$, $\rxconf{c:2neighbour33bigface}$, by Claim~\ref{l:bigface} we can assume that $\S_2^{5+} = \S_{3,3}^{5+}=\S_{4^+, 4^+} = S_{2,4^+}^4 = \emptyset.$ In other words,
   \begin{equation}\label{eq:facetypes}
     \S = \S_3^4 \cup \S_{3, 4^+}^{5+}
   \end{equation}
   and thus $\mathcal{N}_3 = \mathcal{N}_3^*$.

  We consider two subcases. (recall that $a$ is the unique $2$-neighbour of $v$.)
  \vspace{3mm}

  \noindent  {\bf Case 2.1} $a$ is incident to at least one triangular segment.
\\Since $G$ is simple $a$ is incident to exactly one non-triangular segment (denote it by $s$) and one triangular one. By \eqref{eq:facetypes} $s$ has length $4$ and $s$ contains a $3$-neighbour. The $3$-neighbour in $s$ has to be incident to another segment from $\S$ for otherwise we get the excluded configuration \rxconf{c:weak3degree2}. Thus, $|\S|\geq 2$.

    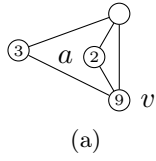
\begin{figure}[!ht]
                    \centering
                    \subfloat[][]{
                    \centering
                    \begin{tikzpicture}[scale=0.95]
                    \tikzstyle{whitenode}=[draw,circle,fill=white,minimum size=8pt,inner sep=0pt]
                    \tikzstyle{blacknode}=[draw,circle,fill=black,minimum size=6pt,inner sep=0pt]
                    \tikzstyle{tnode}=[draw,ellipse,fill=white,minimum size=8pt,inner sep=0pt]
                    \tikzstyle{texte} =[fill=white, text=black]
                         \draw (2,-3) node[whitenode] (v) [label=right:$v$] {\tiny{9}}
                            -- ++(120:0.7cm) node[whitenode] (a) [label=left:$a$] {\tiny{$2$}}
                            -- ++(60:0.7cm) node[whitenode] (x1) [] {}
                            -- ++(200:1.5cm) node[whitenode] (x2) [] {\tiny{$3$}};
                        \draw (x1) edge []  node [label=left:] {} (v);
                        \draw (x2) edge []  node [label=left:] {} (v);
                    \end{tikzpicture}
                    \label{fig:deg91.2}
                }
                \caption{\label{fig:2vertexontriangle} Configuration of $a$ belonging to exactly one triangle}
      \end{figure}

  Note that by \rxconf{c:23triangles-3} $v$ has no more $3$-neighbours. Hence, $v$ sends at most
  $$1 + \tfrac13 + \tfrac12(9-|\S|) \leq 1 + \tfrac13 + \tfrac12(9-2) = 4\tfrac56,$$
  which is smaller than the initial charge. Hence, $v$ has non-negative final charge.
  \vspace{3mm}

  \noindent {\bf Case 2.2}
  $a$ is not incident to any triangular segment.
  \begin{claim}
    $\tfrac32(|\S|-1)\geq n_3$
  \end{claim}
  \begin{proof}
    Consider a subgraph $G'$ of $G$ obtained by removing the vertex $a$. By our assumption $a$ is incident to two non-triangular segments in $G$, which are of length $4$ by \eqref{eq:facetypes}, since by definition $\S_{3,4^+}^{5+}\cap \S_2 = \emptyset$. Then in the process of obtaining $G'$ from $G$ we remove two segments of length $4$, say $s_1,\ s_2$ and create a new segment $s'$, also of length $4$ (see Figure~\ref{fig:9vertexremoving2vertex}).
    By \eqref{eq:facetypes}, $s_1,\ s_2\in \S_3^4(G,v)$
    and hence $s'\in\S_3^4(G',v)$.  In other words
    \begin{equation}\label{eq:podgraphfaces}
     \S_3^4(G', v)\cup \S_{3,4+}^{5+}(G',v) = \S(G,v) \setminus \{s_1,s_2\} \cup \{s'\}.
    \end{equation}

    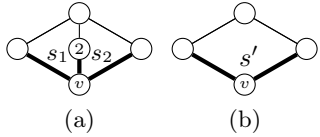
\begin{figure}[ht]
        \centering
        \subfloat[][]{
        \centering
        \begin{tikzpicture}[scale=0.95]
            \tikzstyle{whitenode}=[draw,circle,fill=white,minimum size=8pt,inner sep=0pt]
            \tikzstyle{blacknode}=[draw,circle,fill=black,minimum size=6pt,inner sep=0pt]
            \tikzstyle{smalldot}=[draw,circle,fill=black,minimum size=1pt,inner sep=0pt]
            \tikzstyle{tnode}=[draw,ellipse,fill=white,minimum size=8pt,inner sep=0pt]
            \tikzstyle{texte}=[fill=white, text=black]

            \draw (2,-3) node[whitenode] (v) [] {\tiny{$v$}}
               -- ++(30:1cm) node[whitenode] (x) [] {}
               -- ++(150:1cm) node[whitenode] (y) [] {}
               -- ++(210:1cm) node[whitenode] (z) [] {};

            \draw (2, -2.5) node[whitenode] (a) [] {\tiny{$2$}};

           \draw (x) edge [line width=1.6pt]  node [label=left:] {} (v);
           \draw (y) edge []  node [label=left:] {} (x);
           \draw (z) edge []  node [label=left:] {} (y);
           \draw (v) edge [line width=1.6pt]  node [label=left:] {} (z);

           \draw (a) edge [line width=1.6pt]  node [label=left:] {} (v);
           \draw (y) edge []  node [label=left:] {} (a);

           \node[right=0pt] at (1.4,-2.6) {\footnotesize{$s_1$}};
           \node[right=0pt] at (2.02,-2.6) {\footnotesize{$s_2$}};

          \end{tikzpicture}
          \label{fig:deg9claim7twofaceswithchord}
          }
                              \subfloat[][]{
                              \centering
                              \begin{tikzpicture}[scale=0.95]
                              \tikzstyle{whitenode}=[draw,circle,fill=white,minimum size=8pt,inner sep=0pt]
                              \tikzstyle{blacknode}=[draw,circle,fill=black,minimum size=6pt,inner sep=0pt]
                              \tikzstyle{tnode}=[draw,ellipse,fill=white,minimum size=8pt,inner sep=0pt]
                              \tikzstyle{smalldot}=[draw,circle,fill=black,minimum size=1pt,inner sep=0pt]
                              \tikzstyle{texte}=[fill=white, text=black]

                              \draw (2,-3) node[whitenode] (v) [] {\tiny{$v$}}
                                 -- ++(30:1cm) node[whitenode] (x) [] {}
                                 -- ++(150:1cm) node[whitenode] (y) [] {}
                                 -- ++(210:1cm) node[whitenode] (z) [] {};

                             \draw (x) edge [line width=1.6pt]  node [label=left:] {} (v);
                             \draw (y) edge []  node [label=left:] {} (x);
                             \draw (z) edge []  node [label=left:] {} (y);
                             \draw (v) edge [line width=1.6pt]  node [label=left:] {} (z);

                             \node[right=0pt] at (1.78,-2.6) {\small{$s'$}};

                              \end{tikzpicture}
                              \label{fig:deg9claim7facesconnected}
                              }
                              \caption{\label{fig:9vertexremoving2vertex} Segments in $G$ and in $G'$.}
    \end{figure}

    Then, \eqref{eq:facetypes} and \eqref{eq:podgraphfaces} implies that $\S(G',v) = \S_3^4(G',v)\cup \S_{3,4^+}^{5+}(G',v),$ and hence $\mathcal{N}_3^*(G',v) = \mathcal{N}_3(G',v) = \mathcal{N}_3(G,v)$.

    Since $G$ does not contain configurations \rxconf{c:weak3degree2} and \rxconf{c:2neighbour33bigface}, also in $G'$ vertex $v$ has no weak $3$-neighbours and is not incident to any $4$-face having two semi-weak $3$-neighbours on its boundary (recall that configurations are defined as subgraphs, so if $G'$ contains \rxconf{c:2neighbour33bigface}, so does $G$).
    Therefore, we can apply Claim \ref{l:faces3degree} for $G'$ and $v'=v$.
    Hence, by Claim \ref{l:faces3degree} and equalities \eqref{eq:facetypes} and  \eqref{eq:podgraphfaces} to get the following
    $$|N_3(G,v)| \stackrel{\eqref{eq:facetypes}}{=} |\mathcal{N}_3^*(G',v)| \stackrel{\text{Claim \ref{l:faces3degree}}}{\leq} \tfrac32 |\S_3^4(G', v)\cup \S_{3, 4^+}^{5+}(G', v)| \stackrel{\eqref{eq:podgraphfaces}}{=} \tfrac32 \left( |\S(G,v)| - 1 \right).$$
  \end{proof}
  Vertex $v$ sends $1$ to its $2$-neighbour, $\tfrac13$ to its $3$-neighbours and at most $\tfrac12$ to incident triangular segments. Thus, it sends at most
  $$1 + \tfrac13n_3 + \tfrac12(9-|\S|) \leq  5\tfrac12 +\tfrac13n_3 -\tfrac12(\tfrac23n_3 + 1) = 5$$
  so the final charge of $v$ remains non-negative.
\end{proof}

By lemmas \ref{l:facescharge} and \ref{l:vertexcharge} every face and every vertex has non-negative final charge. Since the total charge does not change when the discharging rules are applied, we obtain that $\sum_{v \in V} (d(v)-4)+\sum_{f \in F}(\ell(f)-4) \ge 0$, thus $4|E|-4|V|-4|F|\ge 0$. However, by Euler's formula $|E|-|V|-|F|=-2$, a contradiction. It follows that the minimal counterexample does not exist, and thus we have proved Theorem~\ref{th:main}.

\refstepcounter{rulecnt}\label{ruleFinal}
\refstepcounter{confcnt}\label{confFinal}

\section{Computer-assisted reducibility}
\label{sec:computer}

In this section we describe the algorithm used for the computer-assisted proof of Lemma~\ref{lem:config}.

The input to the algorithm consists of a graph $H$ that describes a configuration: a graph and a degree function $d:V(H)\rightarrow\mathbb{N}$ that describes degrees of vertices of $H$ in graph $G$, e.g. for \rxconf{c:23triangles-3} $d(x) = 2$, $d(w) = d(z) = 3$ and $d(y) = d(y) = 9$.
Additionally, the input specifies an edge $e_H \in E(H)$, whose meaning will be explained in what follows.

Let us first describe the structure of the proofs generated by our program \texttt{reduce.py}, and next we will elaborate on how we implemented finding such proofs.

\subsection{Proof structure}
\label{sec:structure}

The proof structure is simple. We remove edge $e_H$ from $G$ obtaining a new graph $G'$, colour the graph $G'$ with a colouring $c'$ which exists by the minimality of $G$. Next, we uncolour all the edges of $E(H)$ and the goal is to find a proper colouring $c$ of all edges $G$ (together with $e_H$) such that $c|_{E(G)\setminus E(H)}=c'|_{E(G)\setminus E(H)}$. Note that this proof structure fits the proof of Lemma~\ref{lem:c1} but does not fit the proof of Lemma~\ref{lem:c3}, and in fact configuration~\rxconf{c:2degree2neighbours} cannot be reduced this way.

Of course, even a computer cannot enumerate all colourings of $G'$, since we do not know $G$, and although $H$ is fixed, the number of possible graphs $G$ is infinite. Instead, we enumerate certain {\em classes} of colourings of $G'$.
Such a class is characterized by:
\begin{itemize}
	\item $\Cc=\{C_v \mid v\in V(H)\}$, where $C_v$ is a multiset of colours of the $d_G(v)-d_H(v)$ edges incident with $v$ outside $H$, for every vertex $v\in V(H)$,
	\item a set $\Pp$ that contains a triple $(i, u, v)$ for $i\in\{1,2,3,4\}$ and $u,v\in V(H)$ iff (1) both $u$ and $v$ have exactly one incident edge in $E(G)\setminus E(H)$ colored with $i$ and (2) there is an $i$-colored path in $E(G)\setminus E(H)$ between $u$ and $v$ in the coloring $c'$. 
\end{itemize}
Note that the number of possible sets $\Pp$ is bounded for a fixed configuration $(H,d)$. Hence we are able to enumerate all the classes of colourings of $G'$. Moreover, we claim it is sufficient to know just the class that $c'$ belongs to in order to check whether a colouring of $E(H)$ combined with $c'|_{E(G)\setminus E(H)}$ is a proper colouring of $G$.
Indeed, $\Cc$ is sufficient to make sure that for every vertex of $H$ the number of occurrences of every colour $i$ in the combined colouring is correct (i.e., at most two for $i=1,2,3,4$ and at most one for $i=0$).
Such a colouring could still contain a monochromatic cycle. 
However, if there is such a cycle $C$, it means that there is a colour $i=1,2,3,4$ and a collection $\mathcal{A}$ of $i$-coloured paths in $E(H)$ and another collection $\mathcal{B}$ of $i$-coloured paths in $E(G)\setminus E(H)$ such that $C$ is formed by paths coming alternately from $\mathcal{A}$ and $\mathcal{B}$. Clearly, $\Pp$ is sufficient to guarantee that this does not happen.

Note that the colouring of the inner edges of the configuration, i.e., $c'|_{E(H)\setminus \{e_H\}}$ is not needed here.
However, we do check if there is at least one colouring $c_{\text{inner}}$ of ${E(H)\setminus \{e_H\}}$ that is consistent with $\Cc$ and $\Pp$, because otherwise the pair $(\Cc,\Pp)$ does not correspond to any colour class and we can safely skip it. Here, {\em consistent} means that $c_{\text{inner}}$ can be extended by outer colorings in the class $(\Cc,\Pp)$, i.e., (1) for every $v\in V(H)$, the multiset $C_v \cup \{c_{\text{inner}}(vw) \mid vw \in E(H)\setminus \{e_H\}\}$ has at most two copies of every color and at most one copy of color $0$, (2) for every $(i,u,v)\in \Pp$ we have $i\in C_u\cap C_v$, and (3) after extending the graph $(V(H), {E(H)\setminus \{e_H\}})$ colored with $c_{\text{inner}}$ by colored edges corresponding to paths in $\Pp$ we do not get monochromatic cycles.  

Let us call a reducibility proof with structure as above, a {\em standard proof}.

Recall that in our configurations each vertex has an upper bound on its degree in $G$ (for example 3 for vertex $w$ in $\rxconf{c:23triangles-3}$). For some vertices this upper bound seems not to be specified (for example for vertex $u$ in $\rxconf{c:23triangles-3}$) but then it is just equal to the maximum degree, i.e., 9. These upper bounds are exactly the values of function $d$ that is a part of the input to our program. 
Thus one may think that the program verifies only reducibility for one particular choice of the values of degrees, i.e., the maximal values.
However, this implies reducibility also for any degree function $d':V(G)\rightarrow\mathbb{N}$ such that for each $v\in V(H)$, $d_H(v) \le d'(v) \le d(v)$, as shown in the lemma below.
The intuition behind the proof of this lemma should be clear: a standard reducibility proof for configuration $C$ works also for configuration $C'$.

\begin{lemma}
		Let $C = (H, d)$ and $C'= (H,d')$ be two configurations and assume that for every vertex $v \in V(H)$ we have $d_H(v) \le d'(v) \le  d(v)$. 
		If $C$ is reducible by a standard proof, then $C'$ is reducible.
\end{lemma}

\begin{proof}
	Assume $C'$ appears in a minimal counterexample $G\in \G$. 
	Then we proceed as in the standard reducibility proof for $C$, i.e., we remove from $G$ the same edge $e$ of $H$ which was removed for $C$. 
	Since $G\in \G$, there is a colouring $c':E(G)\rightarrow\{0,1,2,3,4\}$ of $G'-e$.
	This colouring defines sets $\Cc'$ and $\Pp'$, as described above, where $\Cc'=\{C'_w \mid w\in V(H)\}$. For every $v\in V(H)$ define $C_v$ as an arbitrary multiset of $d(v)-d_H(v)$ colors such that
	\begin{equation}
	C'_v \subseteq C_v \subseteq \{0,1,1,2,2,3,3,4,4\}\setminus \{c'(vw)\mid vw\in E(H)\setminus\{e\}\}.
	\label{eq:C_v}
	\end{equation}
	Above, we use the standard notions of inclusion and difference for multisets.
	Note that there is at least one candidate for $C_v$, because $|C'_v|=d'(v)-d_H(v)$ and $d'(v)\le d(v)$. Define $\Cc = \{C_v\mid v \in V(H)\}$.
	
	Recall that in the standard reducibility proof for $C$ we check all classes of colourings that are consistent with at least one colouring of $E(H)\setminus \{e_H\}$. Since $(\Cc',\Pp')$ is consistent with $c'|_{E(H)\setminus \{e_H\}}$, by~\eqref{eq:C_v} also $(\Cc,\Pp')$ is consistent with $c'|_{E(H)\setminus \{e_H\}}$. Hence, the class $(\Cc,\Pp')$ is checked in the standard reducibility proof for $C$, and this proof provides a coloring $c_{\text{inner}}$ of $E(H)$ that can be extended by any coloring in $(\Cc,\Pp')$. But this means that $c_{\text{inner}}$ can be also extended by $c'|_{E(G)\setminus E(H)}$, because no color appears at a vertex more than the required number of times and there are no monochromatic cycles. Hence $G$ can be colored and hence it is not a counterexample, a contradiction.
\end{proof}


\subsection{Implementation}

Script {\tt reduce.py} generates all classes of colourings of $G'$ as follows.
First, the recursive function {\tt generate\_outside\_colorings} generates all possible collections $\Cc$ (as described in Section~\ref{sec:structure}).
At the bottom of the recursion, i.e., when the construction of $\Cc$ is finished, it calls another recursive function {\tt generate\_outer\_paths}.
This function in turn, generates all possible sets $\Pp$.
At the bottom of the recursion both $\Cc$ and $\Pp$ are constructed.
Then, a function {\tt color\_inner} is called, which checks if there is at least one colouring of ${E(H)\setminus {e_H}}$ that is consistent with $\Cc$ and $\Pp$.
If this is not the case, the pair $(\Cc,\Pp)$ is skipped.
Otherwise, the function {\tt color\_inner} is called again, just this time with the parameter {\tt remove\_edge} set to false, which forces the function to colour the edge $e_H$ as well. Function {\tt color\_inner} is a straightforward recursive function which generates all colourings of $E(H)\setminus \{e_H\}$ or $E(H)$ so that the number of edges of given colour at every vertex is as required.
Moreover, after colouring each single edge $e=ab$ it checks if it closes a cycle with a path in $\Pp$ and if not, it updates $\Pp$, i.e., either joins two paths in $\Pp$ (with endpoints in $a$ and $b$, respectively), or extends a path from $\Pp$ (with an endpoint in $a$ or $b$) by $e$, or just adds a single edge path to $\Pp$.

While writing the code, it was our priority to make it as much self-explanatory as possible. We have chosen Python programming language for its high readability of the code. We deliberately have not exploit symmetries between colours to keep the algorithm and the argument simple. The price is the running time: it used 149 hours (on a single processor, Intel Xeon CPU E5-2698 v4, 2.20GHz) for the most demanding configuration \rxconf{c:2weak3neighbours3neighbour}. The time used for other configurations varied from a couple of seconds to 94 minutes.

\section{Conclusion}\label{sect:concl}

We have shown that Conjecture~\ref{con:new} is true for $\Delta=9$ but the $\Delta=7$ is still open. This would be an interesting strengthening of the result of Sanders and Zhao~\cite{DBLP:journals/jct/SandersZ01a} who proved that planar graphs maximum degree 7 are Class I. We were unable to extend our approach to $\Delta=7$ and we suppose that for that a different argument, of a more global nature, is needed. Indeed, when $\Delta=7$ we have less charge at a vertex and we quickly arrive at configurations with elements of negative charge that cannot be reduced using the standard proof generated by our program. Of course it is possible that this can be circumvented by a smart set of discharging rules, but it is worth noting that such a discharging proof is not known for edge coloring (which is a less constrained problem) even for $\Delta=8$.

The Planar Linear Arboricity Conjecture (see Cygan et al.~\cite{CyganHKLW12}) is still open. However, in the spirit of this work one can consider its weaker versions.
For example, what is the maximum $k$ such that edges of a planar graph of maximum degree 8 can be partitioned into $k$ linear forests and $8-2k$ matchings? At the moment, by Vizing Theorem we know that $k\ge 0$, but even the case $k=1$ has not been investigated.

\section*{Acknowledgments}
The authors are grateful to Marek Cygan who implemented an early version of the computer program for checking reducibility in the context of linear arboricity. We thank Arkadiusz Socała for discussions in early stage of the project. We are also very grateful to the reviewers for careful reading and numerous helpful comments. 

This research is a part of projects that have received funding from the European Research Council (ERC)
under the European Union’s Horizon 2020 research and innovation programme Grant Agreements 677651 (ŁK, project TOTAL) and 948057 (ŁK, MP, project BOBR).
JC and ŁK were also supported by the Polish Ministry of Science and Higher Education project Szkoła Orłów, project number 500-D110-06-0465160.

\bibliographystyle{plainurl}
\bibliography{linarb}

\end{document}